\documentclass[reqno,12pt]{amsart}

\usepackage[T1]{fontenc}

\makeatletter
\usepackage{fullpage}

\usepackage{textcmds}  
\usepackage{amsmath, amssymb, amsfonts, amstext, verbatim, amsthm, mathrsfs, stmaryrd}
\usepackage[mathcal]{eucal}
\usepackage[all,cmtip,2cell]{xy}
\usepackage{pgf,tikz,pgfplots}
\pgfplotsset{compat=1.15}
\usetikzlibrary{arrows}
\usetikzlibrary{positioning}
\usetikzlibrary{quotes}
\usepackage{tikz-cd}
\usepackage{multirow}

\usepackage{enumerate}
\usepackage{enumitem}
\usepackage[colorlinks=true,linkcolor=blue,citecolor=blue,urlcolor=blue,citebordercolor={0 0 1},urlbordercolor={0 0 1},linkbordercolor={0 0 1}]{hyperref} 
\usepackage{amsrefs} 
\usepackage[nameinlink]{cleveref}

\usepackage{enotez}
\setenotez{backref=true}

\newcommand{\sbm}[1]{{\let\amp=&\left[\begin{smallmatrix}#1\end{smallmatrix}\right]}}

\def\makeCal#1{%
\expandafter\newcommand\csname c#1\endcsname{\mathcal{#1}}}
\def\makeBB#1{%
\expandafter\newcommand\csname b#1\endcsname{\mathbb{#1}}}
\def\makeFrak#1{%
\expandafter\newcommand\csname f#1\endcsname{\mathfrak{#1}}}
\def\makeRM#1{%
\expandafter\newcommand\csname r#1\endcsname{\mathrm{#1}}}

\count@=0
\loop
\advance\count@ 1
\edef\y{\@Alph\count@}%
\expandafter\makeCal\y
\expandafter\makeBB\y
\expandafter\makeFrak\y
\expandafter\makeRM\y
\ifnum\count@<26
\repeat

\theoremstyle{plain}
\newtheorem{thm}{Theorem}[section]
\newtheorem{cor}[thm]{Corollary}
\newtheorem{lem}[thm]{Lemma}

\newtheorem{prop}[thm]{Proposition}

\theoremstyle{definition}
\newtheorem{rem}[thm]{Remark}
\newtheorem{defn}[thm]{Definition}
\newtheorem{hyp}[thm]{Hypotheses}

\newtheorem{notn}[thm]{Notation}

\newtheorem{ex}[thm]{Example}

\newtheorem*{thm*}{Theorem}
\newtheorem*{prop*}{Proposition}
\newtheorem*{conj*}{Conjecture}

\newtheorem{innercustomthm}{Theorem}
\newenvironment{customthm}[1]
  {\renewcommand\theinnercustomthm{#1}\innercustomthm}
  {\endinnercustomthm}

\def\rm{\mathrm}

\DeclareMathOperator{\rk}{rk}

\DeclareMathOperator{\ch}{ch}
\DeclareMathOperator{\DCoh}{D^b_{coh}}

\DeclareMathOperator{\Coh}{Coh}

\DeclareMathOperator{\Cone}{Cone}

\DeclareMathOperator{\Supp}{Supp}

\DeclareMathOperator{\End}{End}
\DeclareMathOperator{\Ext}{Ext}

\DeclareMathOperator{\GL}{GL}

\DeclareMathOperator{\Hom}{Hom}
\newcommand{\id}{\mathrm{id}}

\DeclareMathOperator{\coker}{coker}
\DeclareMathOperator{\Div}{div}

\newcommand{\pt}{\mathrm{pt}}
\DeclareMathOperator{\QCoh}{QCoh}
\DeclareMathOperator{\rank}{rank}

\newcommand{\RHom}{\mathbf{R}\mathrm{Hom}}

\newcommand*{\sheafhom}{\mathrm{H}\kern -.5pt om}
\newcommand{\Sch}{\mathrm{Sch}}

\DeclareMathOperator{\Spec}{Spec}

\DeclareMathOperator{\Stab}{Stab}

\DeclareMathOperator{\Pic}{Pic}

\DeclareMathOperator{\rdim}{Rdim}
\DeclareMathOperator{\ddim}{Ddim}
\DeclareMathOperator{\gldim}{gldim}
\DeclareMathOperator{\Perf}{Perf}

\DeclareMathOperator{\Ob}{Ob}

\DeclareMathOperator{\im}{im}

\def\db{\mathrm{D}^{\mathrm{b}}}
\def\olk{\overline{k}}

\DeclareMathOperator{\gr}{gr}

\def\bf{\mathbf}

\def\usdim{\overline{\operatorname{Sdim}}}
\def\lsdim{\underline{\operatorname{Sdim}}}
\def\sdim{\operatorname{Sdim}}

\usepackage{pbox}
\usepackage[normalem]{ulem}

\makeatletter

\usepackage{babel}
\begin{document}

\address{Anirban Bhaduri \newline
\indent University of South Carolina, Department of Mathematics\newline
\indent LeConte College, 1523 Greene Street, Columbia, SC 29225}
\email{abhaduri@email.sc.edu}

\address{Isaac Goldberg \newline
\indent Cornell University, Department of Mathematics\newline
\indent 212 Garden Avenue, Ithaca, NY, 14853}
\email{isg26@cornell.edu}

\address{Antonios-Alexandros Robotis \newline
\indent Cornell University, Department of Mathematics\newline
\indent 212 Garden Avenue, Ithaca, NY, 14853}
\email{ar2377@cornell.edu}

\title{Dimension theory of noncommutative curves}

\author{Anirban Bhaduri} 
\author{Isaac Goldberg} 
\author{Antonios-Alexandros Robotis}

\begin{abstract}
    We compute several types of dimension for the bounded derived categories of coherent sheaves of orbifold curves. This completes the calculation of these dimensions for derived categories of noncommutative curves in the sense of Reiten-van den Bergh. Along the way we construct stability conditions for orbifold curves. We also obtain a characterisation of orbifold curves with hereditary tilting bundle in terms of diagonal dimension.
\end{abstract}

\maketitle

\tableofcontents

\section{Introduction}

One conception of noncommutative algebraic geometry is the study of enhanced triang\-ulated categories which are smooth and proper in the suitable sense (see e.g. \cites{OrlovSmoothandproper,ElaginLuntsSchnurer,TabuadavandenBergh,Kuznetsov_2019}). For example, certain invariants of a complex projective variety $X$ can be recovered from its derived category of coherent sheaves, $\DCoh(X)$, illustrating that some geometric properties are intrinsic to the categories themselves. Consequently, it is of interest to have an intrinsic notion of \emph{dimension} of a (pre-)triangulated (dg-)category $\cD$. By now, there are a good number of definitions, which we review in this section (see also \cites{elagin2022calculating,elagin2021three} for more details, including some calculations and conjectures).

In light of these developments, it is reasonable to ask what class of enhanced triangulated categories should be regarded as noncommutative curves. 
In \cite{ReitenvDB}, the authors define a non\-commutative curve as a $k$-linear $\Ext$-finite Abelian category $\cA$ of homological dimension (hd) one, where: 
\[
    \rm{hd}(\cA) := \sup_{E,F\in \cA}\{i\in \bZ:\Hom_{\db(\cA)}(E,F[i]) \ne 0\}. 
\]
If one further imposes that $\cA$ be Noetherian, smooth (i.e. $\db(\cA)$ is saturated as in \cite{BondalKapranov}), and connected, they then provide a classification as follows:

\begin{thm}
\label{T:RvdB}
\cite{ReitenvDB}*{Thm. V.1.2} Let $\cA$ be a Noetherian, smooth, and connected non\-commutative curve, then $\cA$ is equivalent to either 
\begin{enumerate}
    \item $\operatorname{rep} Q$ for $Q$ a finite acyclic quiver; or \vspace{2mm}
    \item $\Coh(\cX)$ for $\cX$ a projective orbifold curve.
\end{enumerate}
\end{thm}

In the present work, we compute various dimensions of $\db(\cA)$ for such noncommutative curves $\cA$. When passing from $\cA$ to $\db(\cA)$ some information is lost. For example, the category of finite dimensional represen\-tations of the Kronecker quiver, $\mathrm{rep}(K_2)$, is not equivalent as an Abelian category to $\Coh(\bP^1)$; nevertheless, there is an exact equivalence $\db(\mathrm{rep}\:K_2) \simeq \DCoh(\bP^1)$ furnished by the Beilinson exceptional collection \cite{Beilinsoncollection}. So, the study of noncommutative curves as Abelian versus triangulated categories is not equivalent. Before stating the main results, we will review the relevant notions of dimension. 

\subsection*{Rouquier dimension}

In seminal work of Rouquier, a notion of dimension for trian\-gulated categories was introduced \cite{Rouquier_2008}. We recall the definition here, following the indexing convention of \cite{BallardFavero}. 

Let $\cT$ be a triangulated category and $\cS$ a subcategory. Denote by $\langle \cS \rangle$ the smallest full subcategory containing $\cS$ which is closed under shifts, isomorphisms, direct coproducts and summands. For any two subcategories $\cS_1$ and $\cS_2$, define an operation $*$ such that $\cS_1 * \cS_2$ is the full subcategory of $\cT$ consisting of all objects $A$ fitting into a distinguished triangle $X_1 \rightarrow A \rightarrow X_2 \rightarrow X_1[1]$ with $X_i \in \cS_i$ for $i=1,2$. Define another operation $\diamond$ on subcategories by $\cS_1 \diamond \cS_2:= \langle \cS_1 * \cS_2 \rangle.$ It is well known that both operations are associative \cite{bondalVdb}. Finally, define inductively: $\langle \cS\rangle_0 := \langle \cS \rangle$ and $\langle \cS \rangle_n := \langle \cS \rangle_{n-1} \diamond \langle \cS \rangle$ for any $n\ge 1$.

For an object $E$ in $\mathcal{T}$, the \emph{generation time} of $E$ is the least $n\in \bN$ such that $\langle E\rangle_n = \cT$, or $\infty$ if no such $n$ exists. $E$ is called a \emph{strong generator} of $\cT$ if it has finite generation time. The \emph{Rouquier dimension} of $\cT$, $\rdim(\cT)$, is the minimal generation time over all $E$. 

In words, Rouquier dimension measures the complexity of a category by the number of cones it takes to build any object out of a single generating object. Strong generators, when they exist, are helpful in understanding the structure of triangulated categories, since they yield exact equivalences to categories of modules over dg-algebras where more tools and techniques are available \cites{bondalVdb,Orlovgeometricrealization,Kellerderivingdg,BondalAssociative}.

We recall several results about Rouquier dimension. First of all, for $X$ a separated scheme of finite type over a perfect field $k$, $\rdim \DCoh(X) <\infty$ \cite{Rouquier_2008}*{Theorem 7.38}. In the same paper it is also proven that for $A$ a finite dimensional $k$-algebra one has $\rdim \db(\mathrm{mod}\:A) \le \gldim(A)$; see \cite{Rouquier_2008}*{Thm. 7.4}. Elagin refines this result by proving that if $Q$ is a finite connected acyclic quiver then $\rdim \db(\rm{mod}\:kQ) = 0$ if and only if $Q$ is of Dynkin type \cite{elagin2022calculating}*{Prop. 4.3}. It follows that if $Q$ is not of Dynkin type then $\rdim\db(\rm{mod}\:kQ) = 1$. 

Given that dimension should be a fundamental invariant of a category, one would like to have explicit computations for a good stock of examples. Orlov proved that $\rdim \DCoh(X) = 1$ for $X$ a smooth projective curve of genus at least one \cite{orlov2008remarks}*{Thm. 6}. On the other hand, $\rdim \DCoh(\bP^1) = 1$ by the derived equivalence to $\db(\mathrm{rep}\:K_2)$ and Rouquier's results. Following this, Orlov conjectured that if $X$ is a smooth quasi-projective scheme of dimension $n$ then $\rdim(X) = n$ \cite{orlov2008remarks}*{Conj. 10}. Few results confirming Orlov's conjecture have been proven to date. In \cite{BallardFavero}*{Cor. 3.22} the conjecture is verified for projective spaces, odd dimensional quadrics, and some Fano threefolds. A natural extension of Orlov's conjecture to orbifold projective lines was proven by Elagin \cite{elagin2022calculating}*{Prop. 6.1}. Recently, the case of normal toric varieties was proven in \cite{FaveroHuang}.

\subsection*{Diagonal dimension}

In what follows, $\cD$ denotes a triangulated category linear over a field $k$ which is equivalent to $\Perf(A)$ for $A$ a smooth and compact dg-algebra over $k$. Given a smooth and compact dg $k$-algebra $A$, the \emph{diagonal dimension} of $A$ is the least $n$ such that there exist $F \in \Perf(A^{\rm{op}})$ and $G\in \Perf(A)$ with the diagonal $A^{\rm{op}}\otimes A$-bimodule $A$ contained in $\langle F\boxtimes G\rangle_n$. This was originally defined in \cite{BallardFavero}*{Def. 2.15} in the geometric context as the least $n$ such that there exist $F,G\in \DCoh(X)$ with $\cO_\Delta \in \langle F\boxtimes G\rangle_n \subset \DCoh(X\times X)$. It is proven in \cite{elagin2022calculating}*{Lem. 4.5} that the two notions are equivalent for $\DCoh(X)$ where $X$ is a smooth projective variety. We denote the diagonal dimension by $\ddim(A)$ in the former case and $\ddim(X)$ in the latter.

Diagonal dimension enjoys some interesting properties and relations to Rouquier dimen\-sion. If $X$ and $Y$ are varieties, then one has $\ddim(X\times Y)\le \ddim(X) + \ddim(Y)$. Furthermore, if $X$ is smooth and proper then $\dim(X) \le \rdim(X) \le \ddim(X) \le 2\dim(X)$ by \cite{BallardFavero}*{Lem. 2.16}; see also \cite{elagin2021three}*{Lem. 4.14}. It is also proven that for $A$ a finite dimensional algebra, $\ddim(A)$ is bounded above by the global dimension of $A$; see \cite{elagin2021three}*{Prop. 4.15}. Furthermore, for $A$ the path algebra of an acyclic quiver $Q$, one has that $\ddim(A) = 1$ by \cite{elagin2022calculating}*{Prop. 4.5}. 

It is also important to note that the inequality $\rdim(X) \le \ddim(X)$ can be strict in some cases. For example, Olander  proves that for $X$ a curve of genus at least one, $\ddim(X) = 2$ \cite{olander2024diagonal}, while $\rdim(X) = 1$ by \cite{orlov2008remarks}.

\subsection*{Serre dimension} 
For $X$ a smooth and projective variety, it is well known that the Serre functor \cite{BondalKapranov} is given by $E\mapsto E\otimes \omega_X[\dim X]$, where $\omega_X$ is the canonical line bundle of $X$. This suggests that one should be able to formulate a notion of dimension for triangulated categories $\cD$ with a Serre functor.

The Serre dimension of $\cD$ can be defined as the categorical entropy of the Serre functor of $\cD$ in the sense of \cite{DHKK}. In \cite{elagin2021three} the notions of upper and lower Serre dimension are introduced which refine the original notion. We recall the definitions here. For two objects $G_1$ and $G_2$ of $\cD$ put 
\[
    e_-(G_1,G_2) = \inf\{i:\Hom^i(G_1,G_2) \ne 0\}\:\: \text{ and }\:\: e_+(G_1,G_2) = \sup\{i:\Hom^i(G_1,G_2)\ne 0\}.
\]
Now, for $G$ and $G'$ \emph{generators} of $\cD$, the \emph{upper} and \emph{lower Serre dimension} of $\cD$ are defined, respectively, as 
\[
\usdim(\cD) = \limsup_{m\to \infty} \frac{-e_-(G,S^m(G'))}{m} \:\:\text{ and }\:\: \lsdim(\cD) = \liminf_{m\to\infty} \frac{-e_+(G,S^m(G'))}{m}.
\]
These definitions do not depend on the choice of generators; see \cite{elagin2021three}*{\S 6}. One has $\lsdim(\cD) \le \usdim(\cD)$ (when $\cD\ne 0$) and in the case where $\lsdim(\cD) = \usdim(\cD)$ we call the common value the \emph{Serre dimension} of $\cD$ and denote it $\sdim(\cD)$.

One advantage of Serre dimension is that it is computable in cases of interest in geometry. Indeed, the Serre dimension of $\DCoh(X)$ for $X$ a smooth and projective variety over a field $k$ is $\dim X$ \cite{elagin2021three}*{Lem. 5.6}. Furthermore, there are some relationships between properties of the space of stability conditions $\Stab(\cD)$ of $\cD$ and the Serre dimension; see \Cref{T:KOT} and \cites{Ikeda_Qiu_2023,Qiu2022,OtaniGldim}. Finally, if $Q$ is a non-Dynkin acyclic quiver then $\sdim \db(\mathrm{rep}\:Q) = 1$. On the other hand, if $Q$ is of ADE type then $\sdim \db(\mathrm{rep}\:Q) = 1 - 2/h$, where $h$ is the Coxeter number of $Q$; see \cite{elagin2022calculating}*{Props. 4.1, 4.2}.

\subsection*{Global dimension}

For a given surjective homomorphism $v:K_0(\cD)\twoheadrightarrow \Lambda$ from the Grothendieck group of $\cD$ to a finite rank lattice, Bridgeland associates to the pair $(\cD,v)$ a complex manifold $\Stab(\cD)$, which has dimension $\rk \Lambda$ if it is nonempty. We will not recall the definition of a Bridgeland stability condition here, referring instead to \cites{Br07,macrischmidt}. We will mention, however, that part of the data of a stability condition is a collection of subcategories of \emph{semistable objects of phase} $\phi$, denoted $\cP(\phi)$, for all $\phi \in \bR$. Given $\sigma \in \Stab(\cD)$, its global dimension is 
\[
    \gldim(\sigma) := \sup\{\phi_2-\phi_1: \Hom_{\cD}(A_1,A_2)\ne 0 \text{ and } A_i \in \cP(\phi_i)\text{ for } i =1,2\}.
\]
The \emph{global dimension} of $\cD$ is then $\gldim(\cD) = \inf_{\sigma \in \Stab(\cD)} \gldim(\sigma)$. It was originally introduced in \cite{Ikeda_Qiu_2023} and is expected in some cases to equal the conformal dimension of a certain Frobenius manifold associated to $\cD$. Given the difficult nature of computing properties of $\Stab(\cD)$ (such as, for instance, non-emptiness) $\gldim(\cD)$ is not known in many cases. 

We recall here some relevant cases where it is known. Qiu has computed the values of $\gldim$ in the case of $\db(\operatorname{rep}Q)$ where $Q$ is a finite acyclic quiver. In the case of a Dynkin quiver $Q$, one has $\gldim(Q) = 1 - 2/h$, where $h$ is the Coxeter number of $Q$; see \cite{Qiu2022}*{Thm. 4.8}. In the case of a non-Dynkin finite acyclic quiver, one has $\gldim(Q) = 1$ by \cite{Qiu2022}*{Thm. 5.2}. Next, by results of \cite{KOT}, it is known that $\gldim(X) =1 $ for $X$ a smooth projective curve. We also record a key theorem: 

\begin{thm}
\label{T:KOT}
\cite{KOT}*{Thm. 4.2} If $\cD$ is equivalent to the perfect derived category of a smooth and proper dg-algebra over $\bC$, then $\usdim(\cD) \le \gldim(\cD)$.
\end{thm}

It is perhaps interesting to remark that in all of the cases considered here this inequality is in fact an equality. The case of orbifolds was the next to be addressed. Recently, Otani has shown using \Cref{T:KOT} that the global dimension of a rational orbifold curve (i.e. one with coarse moduli space $\bP^1$) is one \cite{OtaniGldim}*{Thm. 3.2}.

\subsection*{Results}

In the previous sections we have recalled four notions of dimension for suitable categories $\cD$. In the present work, we compute the values of these dimensions in the case of $\DCoh(\cX)$ for $\cX$ a projective orbifold curve. By the Reiten-Van den Bergh classification of noncommutative curves and the previously cited results, this completes the computation of the dimensions of $\db(\cA)$ for $\cA$ a noncommutative curve. We compile the values in Figure \ref{F:table}. 

\begin{figure}[h]
\label{F:table}
\begin{center}
\begin{tabular}{ |c||c|c|c|c|c| } 
\hline
$\cA$ & $\mathrm{hdim}$ & $\rdim$ & $\ddim$ & $\mathrm{Sdim}$ & $\gldim$\\
\hline
\hline
$\mathrm{mod}(k)$ & 0 & 0 & 0 & 0 & 0 \\ 
\hline
$\mathrm{rep}(Q_{\text{ADE}})$ & 1 & 0 & 1 & $1-\frac{2}{h}$ & $1 - \frac{2}{h}$\\ 
\hline
$\mathrm{rep}(Q_{\text{non-ADE}})$& 1 & 1 & 1 & 1 & 1\\ 
\hline 
\text{rational orbifold curve}& 1 & $\mathbf{1}$ & $\textbf{1 or 2}$ & $\mathbf{1}$ & 1\\
\hline
\text{irrational orbifold curve}& 1 & $\mathbf{1}$ & $\mathbf{2}$ & $\mathbf{1}$ & $\mathbf{1}$ \\
\hline
\end{tabular}
\end{center}
\caption{Above are the various dimensions of $\db(\cA)$ for Abelian $\cA$. Note that here $Q$ denotes a finite acyclic quiver and $\mathrm{rep}(Q)$ the category of its finite-dimensional representations. $Q_{\rm{ADE}}$ denotes a quiver of ADE (Dynkin) type (besides $A_1$). The boldfaced values are results obtained in the present work, while the others follow from the results cited in the previous sections.}
\end{figure}

\newpage
By mildly generalizing the methods of \cite{orlov2008remarks}, we prove:

\begin{customthm}{A}
[=\Cref{T:Orlovtheorem}] The Rouquier dimension of a nonsingular and projective tame orbifold curve $\cX$ is one. 
\end{customthm}

The case of the diagonal dimension is perhaps more surprising. In \cite{olander2024diagonal}, it is proven that the diagonal dimension of a curve $X$ of genus $\ge 1$ is two, while for $\bP^1$ the diagonal dimension is one. The techniques of \cite{olander2024diagonal} apply with only minor modification to show: $\ddim(\cX) = 2$ for $\cX$ a nonsingular projective orbifold curve over an algebraically closed field with coarse moduli space of genus $\ge 1$. In the case where $\cX$ has coarse moduli space $\bP^1$, the result is somewhat different and depends on whether or not the degree of $\omega_{\cX}$ is negative.

The study of $\DCoh(\cX)$ for $\cX$ with coarse moduli space $\bP^1$ was initiated in a different form by Geigle and Lenzing \cite{GL}. They study the representations of certain canonical algebras $A$ that can be presented as $kQ/I$ where $Q$ is a quiver and $I$ is an ideal of relations on its path algebra $kQ$. These algebras can also be obtained as $\End(E)$ for $E$ the tilting bundle on $\cX$ coming from a strong full exceptional collection of line bundles. Consequently, by \cite{BondalAssociative}*{Thm. 6.2} there is a derived equivalence $\DCoh(\cX) \simeq \db(\rm{mod}\:A)$. There is a trichotomy of these orbifold projective lines according to whether or not $\deg(\omega_{\cX})$ is positive, negative, or zero.\footnote{In the literature some authors use instead $\chi_A = -\deg(\omega_{\cX})$ to classify these algebras.}

\begin{customthm}{B}
[=\Cref{T:diagDimOrbCurve}] Let $\cX$ be a smooth and projective orbifold curve over an algebraically closed field. $\ddim(\cX) = 1$ if and only if $\deg(\omega_{\cX}) < 0$. Otherwise, $\ddim(\cX) = 2$. 
\end{customthm}

An interesting consequence is \Cref{C:hereditarytilt}, which states that the condition $\ddim(\cX) = 1$ completely characterises smooth and projective orbifold curves admitting a tilting bundle with hereditary endomorphism algebra. In the last section of the paper, we consider the Serre dimension and global dimensions. The same argument as in \cite{elagin2021three}*{Lem. 5.6} implies $\sdim(\cX) = 1$ and we record this as \Cref{C:sdualitydm}. Finally, we prove: 

\begin{customthm}{C}
[=\Cref{T:stabcond}, \Cref{C:orbifoldcurvegldim}] If $\cX$ is a smooth and projective complex orbifold curve, then there exist Bridgeland stability conditions on $\DCoh(\cX)$ with central charge factoring through the Chen-Ruan cohomology of $\cX$. Furthermore, if the coarse moduli space of $\cX$ has genus at least one, then $\gldim(\cX) = 1$.
\end{customthm}

Recall that in the case of $\cX$ with coarse moduli space $\bP^1$, the global dimension of $\DCoh(\cX)$ is one by \cite{OtaniGldim}.

\subsection*{Acknowledgements}

A.B. would like to thank Matthew Ballard for support and Pat Lank for early discussions and encouragement. A.R. thanks Andres Fernandez Herrero, Daniel Halpern-Leistner, and Hannah Dell for many enlightening conver\-sations and Maria Teresa Mata Vivas for her love and encourage\-ment. Furthermore, we thank Hannah Dell and Pat Lank for comments and corrections on a preliminary version of this paper.

\subsection*{Notation and conventions}
In most parts of the paper, $\cX$ denotes a nonsingular and projective orbifold curve defined over a field $k$ (see \Cref{D:orbifoldcurve} and the subsequent para\-graph). For certain statements these hypotheses are relaxed. The coarse moduli space morphism is denoted $\pi:\cX \to X$ and consequently $X$ is a curve over $k$. Our standard notation is for $\cX$ to have $n$ stacky points, $p_1,\ldots, p_n$, of orders $e_1,\ldots, e_n$ (at least 2), respectively. In \S3, we do not make any hypotheses on the field $k$ but rather that $\cX\to \Spec k$ is \emph{tame}, meaning that the order of the stabilizer groups is coprime to the characteristic of $k$. In \S4, we assume that $k$ is algebraically closed, while in \S5 we assume that $k = \bC$, though this is not strictly necessary for most parts. We sometimes write DM in place of Deligne-Mumford (e.g. DM stack). Throughout, all stacks are defined with respect to the \'{e}tale topology on $\Sch_{/k}$.

\section{Orbifold curves}

For later use, we collect some basic results about orbifold curves. The reader is invited to consult \cites{VoightZB,Taams,Behrend_2014} for more on this topic.

\subsection{Definition and first properties}

First, recall that by the Keel-Mori theorem a DM stack $\cX$ which is separated and finite type over a field $k$ admits a coarse moduli space $\pi:\cX\to X$. See \cites{alper2023stacks,ConradKM} for a modern account of the Keel-Mori theorem or the original article \cite{KeelMori}.

\begin{defn}
\label{D:orbifoldcurve}
    An \emph{orbifold curve} is a dimension one  nonsingular separated DM stack $\cX$ of finite type over a field $k$ with generically trivial stabilizers.
\end{defn}

This is the same as \cite{Taams}*{Def. 1.1.1}, but note that we call such objects orbifold curves rather than stacky curves. It is proven in Lemma 1.1.8 of \textit{loc. cit.} that the coarse moduli space of such an $\cX$ is a nonsingular (separated) curve over $k$. We say $\cX$ is \emph{complete} (or equivalently \emph{projective}) if $X$ is. Next, we recall basic results about orbifold curves. For $e \in \bN$, denote by $\mu_e$ the cyclic group scheme of order $e$. By \cite{Taams}*{Cor. 1.1.33} the automorphism groups of points of $\cX$ are all of the form $\mu_e$.

\begin{prop}
[\cite{Taams}*{Thm. 1.1.7}]
\label{P:localform}
    Let $\cX$ denote an orbifold curve and $p\in \cX$ be a closed point with stabilizer group $\mu_e$. There is an \'{e}tale morphism $V\to X$ with $\pi(p)$ in its image and a Cartesian diagram
    \begin{equation*}
        \begin{tikzcd}
            {[}U/\mu_e{]}\arrow[r]\arrow[d] & \cX\arrow[d,"\pi"]\\
            V\arrow[r]&X
        \end{tikzcd}
    \end{equation*}
    where $U$ is a nonsingular (possibly disconnected) curve over $k$. 
\end{prop}

\begin{rem}
\label{R:atlasbycurve}
By \Cref{P:localform}, for any closed point $p\in \cX$ there exists a nonsingular curve $U$ and an \'{e}tale morphism $U\to \cX$ with $p$ in its image. In particular, $\cX$ admits an \'etale atlas $C\to \cX$ from a nonsingular curve constructed by taking the disjoint union over a finite collection of such $U$.
\end{rem}

\begin{defn}
\label{D:divisor}
    A \emph{Weil divisor} on a stack $\cX$ is a formal $\bZ$-linear combination of irreducible closed codimension one substacks of $\cX$ defined over $k$. 
\end{defn}

When $\cX$ is an orbifold curve, a Weil divisor is a formal linear combination of $k$-rational points, some of which may have nontrivial automorphism groups. We consider sheaves on the \'{e}tale site of $\cX$, where $\cX$ is henceforth assumed to be an orbifold curve. Given $P \in \cX(k)$, we define the sheaf $\cO_{\cX}(P)$ on $\cX_{\text{\'{e}t}}$ as the dual of $\cI_P\subset \cO_{\cX}$, the ideal sheaf of $P$ in $\cX$. For a general Weil divisor $D$, $\cO_{\cX}(D)$ is defined by linearity; one can verify that $\cO_{\cX}(D)$ is a line bundle by passing to an \'{e}tale cover by a scheme.

For a line bundle $L$ on $\cX$, a rational section of $L$ is a section of $L$ over a Zariski open dense subscheme $V$ of $\cX$. Such a $V$ exists because $\cX$ is generically a smooth curve - e.g. on the complement of the stacky points. For any such $V$, we define the \emph{function field} of $\cX$ by $k(\cX) := \cO_{\cX,\xi}$, where $\xi$ is the generic point of $V$. Given a rational section $s$ of $L$, we define the divisor of $s$ by $\Div(s) = \sum_P v_P(s)\cdot P$, where each $P$ is a closed point and $v_P(s) \in \bZ$ is the valuation of $s$ in the \'{e}tale local ring of $\cX$ at $P$.

A Weil divisor is \emph{Cartier} if it is \'{e}tale locally of the form $\Div(f)$ for $f$ a rational function. A pair of Weil divisors $D$ and $D'$ are \emph{linearly equivalent} if there exists a rational function $f \in k(\cX)$ such that $D = D' + \Div(f)$.

\begin{lem}
[\cite{VoightZB}*{Lem. 5.4.5}]
\label{L:stackycurvedivisorlinebundle}
Let $\cX$ denote an orbifold curve.
\begin{enumerate}
    \item Every Weil divisor on $\cX$ is Cartier. 
    \item Every line bundle $L$ on $\cX_{\text{\'{e}t}}$ is isomorphic to $\cO_{\cX}(D)$ for some divisor $D$. $D$ can be taken to be $\Div(s)$, where $s$ is a nonzero rational section of $L$. 
    \item $\cO_{\cX}(D) \cong \cO_{\cX}(D')$ if and only if $D$ and $D'$ are linearly equivalent.
\end{enumerate}
\end{lem}

\begin{defn}
\label{D:degree}
    Let $\cX$ denote an orbifold curve and $p\in \cX(k)$. We define $\deg(p) = 1/\lvert G_p\rvert$, where $G_p$ is the stabilizer group of $p$. For a Weil divisor $D = \sum_{i} n_i\cdot p_i$, we define $\deg(D) = \sum_i n_i \cdot \deg(p_i) \in \bQ$.
\end{defn}

By \Cref{L:stackycurvedivisorlinebundle}, the degree of a divisor $D$ can be just as well computed by taking the divisor of a nonzero rational section of $\cO_{\cX}(D)$. In particular, if $H^0(\cX,\cO_{\cX}(D)) \ne 0$ then $\deg(D) \ge 0$. 

\begin{lem}
\label{L:degdefinition}
    $\deg$ extends uniquely to a homomorphism $\deg:K_0(\cX)\to \bZ$ which is given on the class of a vector bundle $E$ of rank $r$ by $\deg(E) = \deg(\bigwedge^rE)$.
\end{lem}

\begin{proof}
    $K_0(\cX)$ is the Grothendieck group of coherent sheaves on $\cX$, however since $\cX$ satisfies the resolution property \cite{Totaroresolution}*{Thm. 1.2} all elements of $K_0(\cX)$ are of the form $[E]$ for $E$ vector bundle. Thus, it suffices to prove that $\deg$ is additive on short exact sequences of vector bundles. Consider such a sequence $0\to E'\to E\to E''\to 0$ where $E$ is of rank $r = r'+ r''$. Since $\bigwedge^rE \cong \bigwedge^{r'}E'\otimes \bigwedge^{r''}E''$ we have $\deg(E) = \deg(E') + \deg(E'')$. 
\end{proof}

Finally, we recall a result which computes $\pi_*\cO_{\cX}(D)$ by rounding down coefficients.

\begin{prop}
[\cite{VoightZB}*{Lem. 5.4.7}]
\label{P:rounddown}
    Let $\cX$ be a complete orbifold and let $D = \sum_{i=1}^nm_i\cdot p_i$ be a Weil divisor on $\cX$, where the $p_i\in \cX(k)$. $\pi_*\cO_{\cX}(D) = \cO_{X}(\lfloor D\rfloor)$, where 
    \[
    \lfloor D\rfloor = \sum_{i=1}^n \left\lfloor \frac{m_i}{\lvert G_{p_i}\rvert} \right\rfloor \cdot \pi(p_i). 
    \]
\end{prop}

\subsection{Homological properties}

Serre duality holds for tame smooth projective Deligne-Mumford stacks. An elegant and general way of phrasing this is in terms of Serre functors, as introduced in \cite{BondalKapranov}. $-\otimes_{\cO_{\cX}} \omega_{\cX}[n]$ is a Serre functor on $\DCoh(\cX)$, where $\omega_{\cX}$ is the dualizing line bundle of $\cX$ and $n = \dim \cX$; see \cite{GeometricitiyBerghLuntsSchnurer}*{Thm. 6.4}. In particular, one has natural isomorphisms
\[
    \RHom(E,F) \cong \RHom(F,E\otimes \omega_{\cX}[n])^\vee \quad \text{for all } E,F\in \Coh(\cX).
\]
Taking cohomology objects recovers the classical versions of Serre duality.

\begin{lem}
\label{L:homologicaldimension}
If $\cX$ is a tame smooth and projective Deligne-Mumford stack of dimension $n$, then $\Coh(\cX)$ has homological dimension $n$. Moreover, $H^i(\cX,F) = 0$ for all $F\in \Coh(\cX)$ and $i>\dim \cX$.
\end{lem}

\begin{proof}
For all $F,G\in \Coh(\cX)$, $\Ext^k_{\cO_{\cX}}(F,G) \cong \Ext^{n-k}_{\cO_{\cX}}(G,F\otimes \omega_{\cX})^\vee = 0$ for all $k \ge n+1.$ Thus, $\rm{hd} \Coh(\cX) \le n$. On the other hand, $\Ext^n_{\cO_{\cX}}(\cO_{\cX},\omega_{\cX}) = \Hom(\cO_{\cX},\cO_{\cX})^\vee \cong k$ so that $\rm{hd} \Coh(\cX) = n$. For the second claim, apply the natural isomorphism $H^i(\cX,\cF)\cong \Ext^i_{\cO_{\cX}}(\cO_{\cX},\cF)$ for all $i\in \bZ$. 
\end{proof}

\begin{lem}
\label{L:summand}
For an orbifold curve $\cX$, any $E \in \DCoh(\cX)$ is isomorphic to $\bigoplus_{i\in \bZ} H^i(E)[-i]$ where all $H^i(E)\in \Coh(\cX)$ and all but finitely many are zero.
\end{lem}

\begin{proof}
    This follows from the argument of \cite{HuybrechtsFM}*{Cor. 3.15}, and \Cref{L:homologicaldimension}.
\end{proof}

A sheaf $T$ on a Deligne-Mumford stack $\cX$ is called \emph{torsion} if for any \'{e}tale morphism $U\to \cX$ with $U$ a scheme $T(U)$ is a torsion $\cO_{\cX}(U)$-module.

\begin{lem}
\label{L:decomp}
    Suppose $\cX$ is an orbifold curve and consider $E \in \Coh(\cX)$. One has a decomp\-osition $E \cong T\oplus F$ where $T$ is torsion and $F$ is locally free.
\end{lem}

\begin{proof}
    Consider an \'{e}tale atlas $\pi:U\to \cX$ where $U$ is a nonsingular curve. Let $T$ denote the maximal torsion subsheaf of $E$ and put $F:= E/T$. $F$ is pure of dimension one and as a consequence so is $\pi^*(F)$. However, a pure dimension one sheaf on $U$ is torsion free and thus locally free, since the local rings of $U$ are principal ideal domains. Therefore, $F$ is locally free. Next, the dimension of the support of $F^\vee \otimes T$ equals the dimension of the support of $\pi^*(F^\vee \otimes T) = \pi^* F^\vee \otimes \pi^*T$. Since $\pi^*T$ is a torsion sheaf on $U$, it follows that $\dim \Supp (F^\vee\otimes T) = 0$. Because $F$ is locally free, $H^1(\cX,F^\vee\otimes T) \cong \Ext^1(F,T) = 0$ by \Cref{L:homologicaldimension}. Thus, the exact sequence $0\to T\to E\to F\to 0$ is split.
\end{proof}

\subsection{Stability on orbifold curves}

We recall the definition of slope stability for coherent sheaves on an orbifold curve and prove some basic properties that will be used in the sequel.

\begin{defn}
\label{D:generatingsheaf}
Given a tame Deligne-Mumford stack $\cX$, with coarse moduli space $\pi: \cX \to X$, a locally free sheaf $\cE$ is called a \emph{generating sheaf} if for all $\cF\in \QCoh(\cX)$ the adjunction morphism $\pi^*(\pi_* \cH om(\cE,\cF)) \otimes \cE\to \cF$ is surjective. 
\end{defn}

By \cite{Taams}*{Thm. 1.3.13}, every orbifold curve $\cX$ admits a generating sheaf. See also \cite{Kresch}*{Cor. 5.4} for a more general existence result. Taams \cite{Taams} introduces a \emph{standard} generating sheaf on a complete orbifold curve $\cX$ given by 
\begin{equation}
\label{E:generatingsheaf}
        \cE_{\rm{std}} = \bigotimes_{p_i}\bigoplus_{j=0}^{e_i-1} \cO(\tfrac{j}{e_i}\cdot p_i) \oplus \bigotimes_{p_i} \bigoplus_{j=0}^{e_i-1} \cO(\tfrac{-j}{e_i}\cdot p_i).
\end{equation}

As in the case of integral schemes, we define the \emph{rank} of a coherent sheaf $E$ on $\cX$ to be $\rank(E) = \dim_{k(\cX)} E_\xi$ where $\xi$ is the generic point of any open subscheme of $\cX$. For the rest of the paper, we will use the following notation and hypotheses unless otherwise specified:

\begin{notn} 
    \label{N:orbifoldcurve}
        Let $\cX$ denote a complete tame nonsingular orbifold curve and for $n\ge 0$ let $\{p_1,\ldots, p_n\}$ denote the set of stacky points, i.e. the set of closed points of the stack $\cX$ for which the automorphism group is nontrivial (if $n=0$ this set is empty). Let $e_i := \lvert G_{p_i}\rvert$. We will write $\pi:\cX\to X$ for the coarse moduli space map.
\end{notn}

\begin{defn}
\label{D:slopestability}
    We define the \emph{slope} of a nonzero $E\in \Coh(\cX)$ by 
    \[
        \mu(E) = 
        \begin{cases}
            \deg(E)/\rank(E)& \rank(E) > 0 \\
            \infty& \text{else}.
        \end{cases}
    \]
    Here, degree is defined as in \Cref{L:degdefinition}.\footnote{This definition is sensible since the only non-zero rank $0$ sheaves on $\cX$ are torsion sheaves all of which have positive degree by \cite{Taams}*{Thm 1.2.26}.} We call $E$ \emph{stable (resp. semistable)} if for all proper subsheaves $0\subsetneq F\subsetneq E$ one has $\mu(F) < \mu(E)$ (resp. $\mu(F) \le \mu(E)$).
\end{defn}  

The next proposition shows that slope stability as defined in \Cref{D:slopestability} gives rise to Harder-Narasimhan (HN) filtrations as in the classical case.

\begin{prop}
\label{P:stability}
    For every nonzero $F\in \Coh(\cX)$ there is a filtration $0 = F_0 \subsetneq F_1 \subsetneq \cdots \subsetneq F_n = F $ such that $F_i/F_{i-1}$ is semistable for each $1\le i \le n$ and such that $\mu(F_{i+1}/F_{i}) < \mu(F_i/F_{i-1})$ for all $1\le i\le n$.
\end{prop}

\begin{proof}
    Fix a generating sheaf $\cE$ on $\cX$ and a very ample line bundle $\cO_X(1)$ on $X$. Following \cite{olsonStarr}, Nironi introduces a functor $F_{\cE}:\QCoh(\cX)\to \QCoh(X)$ given by $F_{\cE}(G) = \pi_*\cH om(\cE , G)$; see \cite{Nironimodulispaces}. $F_{\cE}$ restricts to a functor $\Coh(\cX)\to \Coh(X)$. Furthermore, Nironi defines the modified Hilbert polynomial of $G \in \Coh(\cX)$ by 
        \[
            P_{\cE}(G,m) = \chi(\cX,G\otimes \cE^\vee \otimes \pi^*\cO_X(m)) = \chi (X , \pi_*\cH om(\cE , G)(m)).
        \]
    Since $X$ is a curve, $P_{\cE}(G,m)$ is a linear polynomial and one can compute by \cite{Taams}*{Thm. 1.3.20} that
    \begin{equation}
        \label{E:hilbpoly}
            P_{\cE}(G,m) = \rk G \cdot \rk \cE \cdot \deg\cO_X(1) \cdot m + d_{\cE}(G) + \rk G \cdot C_\cE
    \end{equation}
    where $d_{\cE}(G) = \deg F_{\cE}(G) - \rk G\cdot \deg F_{\cE}(\cO_{\cX})$ and $C_{\cE}$ is a constant depending only on $\cE$. In \cite{Nironimodulispaces}, stability of coherent sheaves is defined using the slope of the Hilbert polynomial, namely the constant term divided by the coefficient of $m$. However, by \eqref{E:hilbpoly} this defines an equivalent notion of stability as the function $\mu_{\cE}(F) = d_{\cE}(F)/(\rank(\cE)\cdot \rk(F))$. By \cite{Nironimodulispaces}*{Thm. 3.22}, every locally free sheaf $F$ has an HN filtration with respect to $\mu_{\cE}$. Consequently, the proposition will be proven in the case of $F$ locally free if we can produce a generating sheaf $\cE$ such that $\mu_{\cE} = \mu$. Let $\cE = \cE_{\rm{std}}$ as in \eqref{E:generatingsheaf}. We will prove that all $F\in \Coh(\cX)$, $d_{\cE}(F)/\rk(\cE) = \deg(F)$. 
    
    Both $d_{\cE}(F)/\rank(\cE)$ and $\deg(F)$ are additive on short exact sequences. So, it suffices by \cite{Taams}*{Lem. 1.2.20} to prove that $d_{\cE}(L)/\rk(\cE) = \deg(L)$ when $L$ is a line bundle on $\cX$. By \cite{Taams}*{Thm. 1.3.19}, the left hand side is given by the formula:
    \begin{equation}
    \label{E:degree}
        d(L) := d_{\cE}(L)/\rk(\cE) = \deg \pi_* L + \sum_p \frac{1}{e_p} \sum_{i=0}^{e_{p}-1} i\cdot m_{p,i}(L)
    \end{equation}
    where $m_{p,i}(L)$ is the multiplicity of the character $z\mapsto z^i$ in the decomposition of the $i_{p}^*L \in \Coh(B\mu_{e_p})$. To conclude, it will suffice to prove that for any Weil divisor $D$ on $\cX$ and any $p\in \cX(k)$ one has $d(\cO_{\cX}(D+p)) = d(\cO_{\cX}(D)) + d(\cO_{\cX}(p))$ and that $d(\cO_{\cX}(p)) = \deg(p)$. By \eqref{E:degree} it is clear that the first claim holds when $p$ and $D$ have disjoint support. Thus, we may assume that $D = n\cdot p$ for $n\in \bZ$. Now, $i_p^*(\cO_{\cX}(np))$ is the character of $\mu_{e_p}$ given by $z\mapsto z^n$. Write $n = m\cdot e_p +r$ where $0\le r\le e_p - 1$ and $m\in \bZ$. By \Cref{P:rounddown}, $\pi_*\cO_{\cX}(n\cdot p) = \cO_{X}(m\cdot \pi(p))$. On the other hand, $m_{r,p}(\cO_{\cX}(n\cdot p)) = 1$ and all other $m_{i,p}$ are zero. Thus, $d(\cO_{\cX}(n\cdot p)) = m + \frac{r}{e_p} = \frac{n}{e_p}$; this proves both of the claims.

    Finally, let $E$ be an arbitrary coherent sheaf. By \Cref{L:decomp} there is a decomposition $E = F\oplus T$ for $F$ locally free and $T$ torsion. Consider the HN filtration $0 = F_0 \subsetneq F_1\subsetneq \cdots \subsetneq F_n = F$ of $F$. The HN filtration of $E$ is $0 = E_{-1} \subsetneq E_0\subsetneq \cdots \subsetneq E_n = E$ where $E_i := F_i \oplus T$. 
\end{proof}

\begin{prop}
\label{P:propertiesofsemistable}
    All sheaves are on $\cX$ as in \Cref{N:orbifoldcurve}.
    \begin{enumerate}
        \item If $L$ is a line bundle then $L$ is semistable.
        \item If $F$ is a semistable vector bundle then $F^\vee$ is semistable and $\mu(F^\vee) = -\mu(F)$. 
        \item If $F$ is a semistable vector bundle and $L$ is a line bundle then $F\otimes L$ is semistable with $\mu(F \otimes L) = \mu(F) + \frac{\deg(L)}{\rk F}$.
        \item If $F$ is a semistable vector bundle and $\mu(F) < 0$ then $H^0(\cX,F) = 0$.
    \end{enumerate}
\end{prop}

\begin{proof}
    (1) $\deg L = \mu(L)$ and a subsheaf $L'\subset L$ is necessarily torsion free of rank $1$, hence also a line bundle. The inclusion $L'\hookrightarrow L$ is a nonzero morphism and consequently $\deg L' \le \deg L$. (2) is by $\deg L^\vee = -\deg L$ and the fact that formation of the determinant bundle commutes with dualization. (3) follows from the fact that if $F$ is of rank $r$ then $\bigwedge^r(F\otimes L) \cong \bigwedge^r(F) \otimes L$. (4) A global section of $F$ is equivalent to a nonzero morphism of sheaves $\cO \to F$ and the existence of a subsheaf $F'$ of $F$ which is quotient of $\cO$. Thus, $\mu(F) < 0 = \mu(\cO) \le \mu(F')$ so $F$ is not semistable. 
\end{proof}

\begin{lem}
\label{L:ssvanishing}
    If $F$ is a semistable vector bundle on $\cX$ with $\mu(F) \ge 2g + n$ then for any line bundle $L$ with $\deg(L) \ge -1$ one has $H^1(\cX, F\otimes L) = 0$.
\end{lem}

\begin{proof}
    By Serre duality, $H^1(\cX,F\otimes L) \cong H^0(\cX,F^\vee \otimes L^\vee \otimes \omega_{\cX})^*$. By \Cref{P:propertiesofsemistable}, $F^\vee \otimes L^\vee \otimes \omega_{\cX}$ is semistable. By \cite{VoightZB}*{Prop. 5.5.6}, $\deg \omega_{\cX} = 2g(X) - 2 + \sum_{i=1}^n \tfrac{e_i-1}{e_i}$. By \Cref{P:propertiesofsemistable}, $\deg(F^\vee \otimes L^\vee \otimes \omega_{\cX})<0$ and the result follows.
\end{proof}

\subsection{Global Generation}
We consider the diagram
\[
\begin{tikzcd}
    \cX\arrow[r,"\pi"]\arrow[dr,"f",swap]&X\arrow[d,"g"]\\
    &\Spec k.
\end{tikzcd}
\]
where $f$ and $g$ are the structure morphisms and $\pi$ is the coarse moduli map. Recall that a coherent sheaf $\cF$ on a scheme $X$ over a field $k$ is globally generated if and only if the counit of adjunction $g^*g_* \cF\to \cF$ is surjective. 

\begin{defn} 
\label{D:generated}
    Fix $\cE$ a generating sheaf on $\cX$. Call $F\in \Coh(\cX)$ $\cE$-\emph{generated} if there exists $r\in \mathbb{N}$ and a surjection $\cE^{\oplus r} \twoheadrightarrow \cF$.
\end{defn}

For $X$ a scheme, a coherent sheaf $F$ on $X$ is globally generated in the usual sense if it is $\cO_X$-generated as in \Cref{D:generated}.

\begin{lem}
\label{L:ggonstack}
    Consider $F\in \Coh(\cX)$. 
    \begin{enumerate} 
        \item If $\pi_*(\cE^\vee \otimes F)$ is globally generated then $\Hom(\cE,F)\otimes \cE \to F$ is surjective. In particular, $F$ is $\cE$-generated.
        \item If in addition $\Hom(\cE,\cO_{\cX}) \otimes \cE \to \cO_{\cX}$ induces a surjection on global sections, then $\Hom(\cE,F)\otimes F \to F$ induces a surjection on global sections.
    \end{enumerate}
\end{lem}

\begin{proof}
    (1) By $\cE$ being a generating sheaf, $\pi^*(\pi_*(\cE^\vee \otimes F))\otimes \cE\twoheadrightarrow F$ is surjective. Since $\pi_*(\cE^\vee \otimes F)$ is globally generated $g^*g_*\pi_*(\cE^\vee \otimes F)\twoheadrightarrow \pi_*(\cE^\vee \otimes F)$ is surjective. $\pi^*$ is right exact, so $\pi^*g^*g_*\pi_*(\cE^\vee \otimes F)\twoheadrightarrow \pi^*\pi_*(\cE^\vee \otimes F)$ is surjective and tensoring by $\cE$ gives $\Hom(\cE,F)\otimes \cE = f^*f_*(\cE^\vee \otimes F)\otimes \cE \twoheadrightarrow \pi^*\pi_*(\cE^\vee \otimes F)\otimes \cE \twoheadrightarrow F$. For (2), consider $s\in H^0(\cX,F)$. There is a commutative square 
    \[
    \begin{tikzcd}
        \Hom(\cE,\cO_{\cX})\otimes \cE \arrow[r]\arrow[d]& \cO_{\cX}\arrow[d,"s"]\\
        \Hom(\cE,F)\otimes \cE\arrow[r]& F.
    \end{tikzcd}
    \]
    Applying $H^0$, we see that $\Hom(\cE,\cO_{\cX})\otimes H^0(\cX,\cE)\to H^0(\cX,F)$ maps onto the subspace containing $s$. Consequently, so does $\Hom(\cE,F)\otimes H^0(\cX,\cE)\to H^0(\cX,F)$.
\end{proof}

\subsection{Chen-Ruan cohomology of orbifold curves and orbifold Chern character}

Let $\cX$ denote a nonsingular and projective orbifold curve over $\bC$. The \emph{Chen-Ruan cohomology} of $\cX$ is $H^*_{\rm{CR}}(\cX;\bZ):= H^*(\lvert I_{\cX}\rvert;\bZ)$, where $I_{\cX}$ is the inertia stack of $\cX$ and $\lvert I_{\cX}\rvert$ denotes its coarse space, regarded as a complex analytic space by analytification \cite{GAGA}. The reader is invited to consult \cites{ChenRuan,AGV} for a more thorough discussion of Chen-Ruan cohomology.

\begin{ex}
\label{Ex:CRofcurve}
    We use \Cref{N:orbifoldcurve}. $I_{\cX}$ is a disjoint union of stacks: 
    \[
    I_{\cX} = \cX \sqcup \bigsqcup_{i=1}^n \bigsqcup_{e_i-1}B\mu_{e_i} \text{ and } \lvert I_{\cX}\rvert = X \sqcup \bigsqcup_{i=1}^n \bigsqcup_{e_i-1} \Spec \bC.
    \]
    In particular, $H^*_{\rm{CR}}(\cX) = H^*(X;\bZ) \oplus \bigoplus_{i=1}^n \bigoplus_{e_i-1}H^*(\pt; \bZ)$.
\end{ex}

Next, we define a Chern character for orbifold curves. Associated to each of the $p_i$ is a residual gerbe, $B\mu_{e_i}$ and there is a pullback morphism $i_{p_i}^*:\Coh(\cX)\to \Coh(B\mu_{e_i})$ which sends $E\in \Coh(\cX)$ to its corresponding representation of $\mu_{e_i}$. Being a representation of $\mu_{e_i}$, $i_{p_i}^*(E)$ decomposes as a sum of irreducible representations, each of which corresponds to a character $\chi_j:\mu_{e_i}\to \bG_m$, for $0\le j \le e_i-1$. Let $m_{p_i,j}(E)$ denote the multiplicity of the character $\chi_j$ of $\mu_{e_i}$ in $i_{p_i}^*(E)$. Put $\mathfrak{m}_{p_i}(E) = (m_{p_i,1}(E),\ldots,m_{p_i,e_i-1}(E)) \in \bZ^{e_i-1}$. Since $\cX$ has the resolution property by \cite{Totaroresolution}*{Thm. 1.2}, it suffices to define $\alpha:K_0(\cX) \to \Pic(X) \oplus \bZ \oplus \bigoplus_{i=1}^n \bZ^{e_i-1}$ on vector bundles. So, given (the class of) a vector bundle, $E$, we define
\[
    \alpha(E) = (\det(\pi_*(E)),\rk(E), \mathfrak{m}_{p_1}(E),\ldots, \mathfrak{m}_{p_n}(E)).
\]

\begin{prop}
\label{P:K0computation}
\cite{Taams}*{Thm. 1.2.27} The map $\alpha:K_0(\cX) \to \Pic(X) \oplus \bZ \oplus \bigoplus_{i=1}^n \bZ^{e_i-1}$ is an isomorphism.
\end{prop}

\begin{defn} 
\label{D:chorb}
Let $\cX$ denote a nonsingular projective orbifold curve over $\bC$. The \emph{orbifold Chern character}, $\ch_{\rm{orb}}:K_0(\cX) \to H_{\rm{CR}}^*(\cX)$, is defined by the following diagram:

\[
\begin{tikzcd}K_0(\cX)\arrow[r,"\ch_{\rm{orb}}"]\arrow[d,"\alpha",swap] & H^*_{\rm{CR}}(\cX) \arrow[d,equal]\\
    \Pic(X)\oplus \bZ \oplus \bigoplus_{i=1}^n \bZ^{e_i-1}\arrow[r,"(\ch {,}\id)"] & H^*(X;\bZ) \oplus \bigoplus_{i=1}^n \bZ^{e_i-1}.
\end{tikzcd}
\]
Here, $\ch:K_0(X) = \Pic(X)\oplus \bZ \to H^*(X;\bZ)$ is the usual Chern character defined by $E\mapsto (\rk(E),\deg(E))$. The right vertical identification is furnished by \Cref{Ex:CRofcurve}. Later, we will consider the \emph{algebraic} Chen-Ruan cohomology, $H^*_{\rm{CR,alg}}(\cX) := \im \ch_{\rm{orb}}$. In our case, $H^*_{\rm{CR,alg}}(\cX) = H^0(X;\bZ)\oplus H^2(X;\bZ) \oplus \bigoplus_{i=1}^n \bZ^{e_i-1}$.
\end{defn}

\section{Rouquier dimension of orbifold curves}

This section is devoted to the proof of the following theorem, which generalises \cite{orlov2008remarks}*{Thm 6.}. We use \Cref{N:orbifoldcurve} throughout.

\begin{thm}
\label{T:Orlovtheorem}
    If $\cX$ is a complete tame orbifold curve over a field $k$, then $\rdim \DCoh(\cX) = 1$.
\end{thm}

By \cite{alper2023stacks}, the nonsingularity hypotheses on $\cX$ and $X$ imply that $\pi^*$ is exact. By \cite{elagin2022calculating}*{Prop. 6.1}, if the coarse moduli space of $\cX$ is $\bP^1$, then $\rdim \DCoh(\cX) = 1$. So we may assume that $g := g(X) \ge 1$. By \cite{BallardFavero}*{Lem. 2.17}, $\rdim \DCoh(\cX)\ge \rdim\DCoh(X)$ for any tame $\cX$ with $X$ reduced and separated. To prove \Cref{T:Orlovtheorem}, we construct $H\in\DCoh(\cX)$ such that $\langle H\rangle_1 = \DCoh(\cX)$. Since $\langle H\rangle_1$ is closed under shifts, isomorphisms, sums, and summands, it suffices by \Cref{L:summand} to prove that $E \in \langle H\rangle_1$ for all $E\in \Coh(\cX)$. Writing $E = F \oplus T$ as in \Cref{L:decomp}, we are further reduced to considering the cases where $E$ is locally free or torsion separately.

\begin{hyp}
\label{Hyp}
    We assume $g = g(X) \ge 1$ and consider a line bundle $L$ on $X$ with $\deg(L) \ge 8g+4n$. $\cE$ denotes the locally free sheaf $\cE = \bigoplus_{i=1}^n \bigoplus_{j=0}^{e_i-1}\cO_{\cX}(-j\cdot p_i)$, which is a generating sheaf by \cite{Taams}*{Thm. 1.3.13}. We put 
    \[
        H = (\pi^*L^{-1}\otimes \cE) \oplus \cE\oplus (\pi^*L\otimes \cE^\vee) \oplus (\pi^*L^2\otimes \cE^\vee).
    \]
\end{hyp}

We prove that $H$ is a 1-step generator of $\DCoh(\cX)$ by modifying the arguments of \cite{orlov2008remarks}. If $\cX\to \Spec k$ is an algebraic stack, $\cX_{\overline{k}}$ denotes its base-change along $\Spec \overline{k}\to \Spec k$, where $\overline{k}$ is the algebraic closure of $k$. If $F\in \QCoh(\cX)$, we write $F_{\overline{k}}$ for its pullback to $\cX_{\overline{k}}$. Formation of the coarse moduli space commutes with flat base-change and by abuse of notation we also write $\pi:\cX_{\overline{k}}\to X_{\olk}$ for the base-changed morphism.

\begin{lem}
\label{L:vanishingposdeg}
    For a coherent sheaf $U$ on $\cX$, $H^1(\cX,U) = 0$ if and only if $H^1(\cX,\cE^\vee \otimes U) = 0$.
\end{lem}

\begin{proof}
    By \Cref{Hyp}, $\cE^\vee$ is a sum of line bundles of the form $\cO_{\cX}(j\cdot p_i)$ with $j\ge 0$. It suffices to prove that $H^1(\cX,U) = 0$ implies $H^1(\cX,\cE^\vee \otimes U) = 0$, the other implication being immediate from the fact that $H^1(\cX,U)$ is a summand of $H^1(\cX,\cE^\vee \otimes U)$. The case where $U$ is torsion is immediate and so by \Cref{L:decomp} we may assume that $U$ is locally free. By \cite{Taams}*{Ex. 1.2.24}, there is a short exact sequence $0\to \cO_{\cX} \to \cO_{\cX}(j\cdot p_i) \to T\to 0$, where $T$ is a torsion sheaf. Tensoring by $U$ gives $0\to U \to U\otimes \cO_{\cX}(j\cdot p_i) \to T'\to 0$, where $T'$ is again torsion. We obtain an exact sequence $H^1(\cX,U)\to H^1(\cX, U\otimes \cO_{\cX}(j\cdot p_i)) \to H^1(\cX,T') \to 0$. Since $T'$ is torsion, $H^1(\cX,T') = 0$ and therefore $H^1(\cX, U\otimes \cO_{\cX}(j\cdot p_i)) = 0$. 
\end{proof}

\begin{lem}
\label{L:vanishinggg}
    Let $F$ be a vector bundle on $\cX$. Suppose that for all line bundles $M$ on $X_{\overline{k}}$ with $\deg M =d$, one has $H^1(\cX_{\overline{k}},\cE^\vee_{\olk} \otimes F_{\olk}\otimes \pi^*M) = 0$. Then
    \begin{enumerate}
        \item $H^1(\cX, \cE^\vee \otimes F \otimes \pi^*N) = 0$ for any line bundle $N$ on $X$ with $\deg N \ge d$; and 
        \item $F\otimes \pi^*N$ is $\cE$-generated for any line bundle $N$ with $\deg N > d$. 
    \end{enumerate}
\end{lem}

\begin{proof}
    Suppose $k = \overline{k}$. $\pi_*(\cE^\vee \otimes F) = T \oplus E$ for $T$ torsion and $E$ locally free. By hypothesis, $0 = H^1(\cX, \cE^\vee \otimes F \otimes \pi^*M) = H^1(X,\pi_*(\cE^\vee \otimes F) \otimes M) = H^1(X,E\otimes M)$ for any line bundle $M$ of degree $d$. Therefore, by \cite{orlov2008remarks}*{Lem. 9} for any line bundles $N_1$ with $\deg(N_1)\ge d$ and $N_2$ with $\deg(N_2)>d$ one has that $H^1(X,E\otimes N_1) = 0$ and that $E\otimes N_2$ is generated by global sections. (1) now follows. $T\otimes N_2$ is a torsion sheaf and thus generated by global sections. So, $\pi_*(\cE^\vee \otimes F)\otimes N_2$ is generated by global sections and (2) follows from \Cref{L:ggonstack}. The case where $k$ is not algebraically closed follows from flat base change.
\end{proof}

\begin{lem}
\label{L:torsionsheafsurjection}
    For any torsion sheaf $T$ on $\cX$, there is an exact sequence 
    \[
        (\pi^*L^{-1}\otimes \cE)^{\oplus r_1}\to \cE^{\oplus r_0}\to T\to 0.
    \]
\end{lem}

\begin{proof}
    $T\otimes \cE^\vee$ is torsion and so $T$ is $\cE$-generated by \Cref{L:ggonstack}. Consider a short exact sequence $0\to U \to \cE^{\oplus r_0}\xrightarrow{\alpha} T\to 0$ such that $H^0(\alpha)$ is surjective. We claim that $U\otimes \pi^*L$ is $\cE$-generated, whence the result follows. By \Cref{L:vanishingposdeg} and \Cref{L:vanishinggg}, it suffices to show that $H^1(\cX, U\otimes \pi^*N) = 0$ for all line bundles $N$ with $\deg(N) = \deg(L) - 1$. Applying $\pi_*$ and the projection formula to $0 \to U\otimes \pi^*N\to \cE^{\oplus r_0}\otimes \pi^*N\to T\otimes \pi^*N \to 0$ yields
    \[
    0\to \pi_*(U) \otimes N \to \pi_*(\cE^{\oplus r_0}) \otimes N \xrightarrow{\pi_*(\alpha) \otimes \id} \pi_*(T)\otimes N \to 0.
    \]
    $\pi_*(\alpha)$ is surjective on global sections because $\alpha$ is and since $T$ is a torsion sheaf it follows that $\pi_*(\alpha)\otimes \id$ is also surjective. Since $\pi_*\cE$ is a direct sum of line bundles of degree at least $-1$, $H^1(X,\pi_*(\cE^{\oplus r_0})\otimes N) = 0$ for degree reasons, so $H^1(\cX,U\otimes \pi^*N) = H^1(X,\pi_*(U)\otimes N) = 0$. 
\end{proof}

\begin{lem}
\label{L:twosections}
    Let $X$ denote a smooth proper curve of geometric genus $g\ge 1$ over an infinite field $k$. If $L$ is a line bundle on $X$ of degree at least $2g$ then there is a surjective morphism of sheaves $\cO_X^{\oplus 2}\twoheadrightarrow L$. 
\end{lem}

\begin{proof}
    By Riemann-Roch, $h^0(X,L) \ge 2$. $L$ is base-point free by \cite{HartshorneAG}*{Cor. IV.3.2} and defines a map $X\to \bP^n$ for some $n\ge 1$. $L$ has at least two linearly independent global sections and thus the map is nonconstant. A general hyperplane $H$ in $\bP^n$ intersects $X$ at $\deg L$ many points $\{P_i\}$. Choose another hyperplane $H'$ avoiding $\{P_i\}$. The equations of $H$ and $H'$ correspond to global sections $s,s'$ of $L$ which do not vanish simultaneously and thus define a surjection $\cO_X^{\oplus 2} \twoheadrightarrow L$.
\end{proof}

The next result is a modification of \cite{orlov2008remarks}*{Lem. 8}. The proof goes through with only minor modifications, but we include it for the convenience of the reader.

\begin{lem}[Orlov's Main Lemma] 
\label{L:Orlovmainlemma}
Let $L$ be a line bundle on $X$ with $\deg L \ge 8g+4n$ and let $F$ denote a vector bundle on $\cX$ with Harder-Narasimhan filtration $0 = F_0\subset F_1\subset \cdots \subset F_n = F$. Suppose there exists a maximal $i$ such that $\mu(F_i/F_{i-1}) \ge 4g+2n$. There are exact sequences 
\begin{enumerate}[label=(\alph*)]
    \item $(\pi^*L^{-1}\otimes \cE)^{\oplus r_1}\xrightarrow{\alpha} \cE^{\oplus r_0}\to F_i\to 0$
    \item $0\to F/F_i \to (\pi^*L\otimes \cE^\vee)^{\oplus s_0}\xrightarrow{\beta} (\pi^*L^2\otimes \cE^\vee)^{\oplus s_1}.$
\end{enumerate}
If $\mu(F_1)<4g+2n$ then there is an exact sequence $0\to F \to (\pi^*L\otimes \cE^\vee)^{\oplus s_0}\xrightarrow{\beta} (\pi^*L^2\otimes \cE^\vee)^{\oplus s_1}.$ 
\end{lem}

\begin{proof}
    Suppose first that $k = \overline{k}$ and consider (a). By \Cref{L:ssvanishing}, if $G$ is a semistable vector bundle on $\cX$ with $\mu(G) \ge 2g+n$ then $H^1(\cX,G\otimes L) = 0$ for all $L$ with $\deg L \ge -1$. So, by \Cref{L:vanishingposdeg} and \Cref{L:vanishinggg} it follows that $G$ is $\cE$-generated. In particular, $F_i$ is $\cE$-generated, being an iterated extension of semistable vector bundles with slope at least $4g+2n$. So, we may construct a short exact sequence
    \begin{equation}
    \label{E:ses2}
        0\to U  \to \cE^{\oplus r_0}\xrightarrow{\varphi} F_i \to 0
    \end{equation}
    such that $H^0(\varphi)$ is surjective since $\cO_{\cX}$ is a summand of $\cE$. Next, consider a line bundle $M$ on $X$ of degree $2g+n$. By \Cref{L:twosections}, there is a short exact sequence 
    \[
    0\to M^\vee\xrightarrow{(-s'{,}s)^t}\cO_{X}^{\oplus 2} \xrightarrow{(s{,}s')} M \to 0
    \]
    which pulls back to 
    \begin{equation}
        \label{E:ses3}
        0\to \pi^* M^\vee\to \cO_{\cX}^{\oplus 2} \to \pi^*M \to 0.
    \end{equation}
    Forming the tensor product of short exact sequences, \eqref{E:ses2} $\otimes$ \eqref{E:ses3} gives a commutative diagram with exact rows and columns: 
    \begin{center}
        \begin{tikzcd}
            & 0 \arrow[d] & 0 \arrow[d] & 0 \arrow[d]  & \\
            0 \arrow[r] & U \otimes \pi^*M^\vee \arrow[r] \arrow[d] & \cE^{\oplus r_0} \otimes \pi^* M^\vee \arrow[r] \arrow[d]  & F_i \otimes \pi^*M^\vee \arrow[r] \arrow[d] & 0 \\
            0 \arrow[r] & U^{\oplus 2} \arrow[r] \arrow[d] & \cE^{\oplus 2 r_0} \arrow[r,"\varphi^{\oplus 2}"] \arrow[d,"\chi"]& F_i^{\oplus 2} \arrow[r] \arrow[d,"\psi"] & 0 \\
            0 \arrow[r] & U \otimes \pi^*M \arrow[r] \arrow[d] & \cE^{\oplus r_0} \otimes \pi^* M \arrow[r,"\upsilon"] \arrow[d] & F_i \otimes \pi^* M \arrow[d] \arrow[r]  & 0 \\
             & 0 & 0 & 0 &
        \end{tikzcd}
    \end{center}
    Since $F_i\otimes \pi^*M^\vee$ is an extension of semistable vector bundles with slope at least $2g+n$, $H^1(\cX,F_i \otimes \pi^*M^\vee) = 0$ and so $H^0(\psi)$ is surjective. Also, by construction, $H^0(\varphi^{\oplus 2})$ is surjective and so by commutativity, $H^0(\upsilon)$ is surjective. Finally, as $H^1(\cX,\cE^{\oplus r_0} \otimes \pi^* M) = 0$ for degree reasons, it follows that $H^1(\cX,U \otimes \pi^*M) = 0$. So, by \Cref{L:vanishingposdeg} and \Cref{L:vanishinggg}, $U\otimes \pi^*L$ is $\cE$-generated and we obtain an exact sequence $\pi^*L^{-1}\otimes \cE^{\oplus r_1}\xrightarrow{\alpha} \cE^{\oplus r_0}\to F_i \to 0$. This proves (a) in the algebraically closed case. The general case is deduced using flat base change and the observation that we only applied results involving vanishing of cohomology groups. The dual statement (b) is proven in \Cref{L:Orlovlemmab}.
\end{proof}

\begin{proof}[Proof of \Cref{T:Orlovtheorem}]
    By \Cref{L:summand} and \Cref{L:decomp} it suffices to prove that 
    \begin{enumerate}
        \item any torsion sheaf $T$ in the heart $\Coh(\cX)$ is in $\langle H\rangle_1$; and 
        \item any locally free sheaf $F$ in the heart $\Coh(\cX)$ is in $\langle H\rangle_1$.
    \end{enumerate}
    For (1), consider a torsion sheaf $T$. By \Cref{L:torsionsheafsurjection}, we have an exact sequence $(\pi^*L^{-1}\otimes \cE)^{\oplus r_1}\to \cE^{\oplus r_0}\to T \to 0$. By \Cref{L:summand} it follows that $T$ is a summand of $\Cone((\pi^*L^{-1}\otimes \cE)^{\oplus r_1}\to \cE^{\oplus r_0})$ and consequently lies in $\langle H \rangle_1$. For (2), using \Cref{L:Orlovmainlemma}(a,b) and \Cref{L: orlovCone} it follows that $F$ is a summand of the cone of a map 
    \[
    \begin{tikzcd}
      (\pi^* L^{-1}\otimes \cE)^{\oplus r_1} \oplus (\pi^* L\otimes \cE^\vee)^{\oplus s_0}[-1] \arrow[r,"\sbm{\alpha\amp \varepsilon \\ 0 \amp \beta}"]& \cE^{\oplus r_0}\oplus (\pi^* L^2\otimes \cE^\vee)^{\oplus s_1}[-1] 
    \end{tikzcd}
    \]
    and thus $F\in \langle H\rangle_1$. 
\end{proof}

\section{Diagonal dimension of orbifold curves}

Recall the following main result from the paper of Olander \cite{olander2024diagonal}.

\begin{thm}
\label{T:diagDimCurve}
Let $X$ be a smooth projective curve over a field $k$. Then 
\begin{align*}
\ddim(X) = \begin{cases}
    1 & \text{if } h^1(X, \cO_X ) = 0\\
    2 & \text{if } h^1(X, \cO_X ) \ge 1
\end{cases}
\end{align*}
\end{thm}

This is to say that of the two possible values for diagonal dimension of a smooth curve provided by \Cref{L:dimbound}, the answer depends on the curve's arithmetic genus. In this section we explain the modifications needed to generalize this result to the setting of orbifold curves. For simplicity, in this section we assume $k$ is algebraically closed.

As explained in \S1, for $X$ a scheme, $\ddim(X) = \min\{n: \cO_{\Delta} \in \langle F\boxtimes G\rangle_n, F,G \in \DCoh(X)\}$. For schemes, $\cO_\Delta$ can equivalently be thought of as the pushforward of $\cO_X$ along the diagonal morphism $\Delta : X \to X \times X$, or as the structure sheaf of the diagonal subscheme of $X \times X$. For stacks, these two notions may no longer agree as $\Delta$ is not in general even a monomorphism. 

\begin{defn}
    Let $\cX$ be an algebraic stack, and denote the diagonal morphism by $\Delta : \cX \to \cX \times \cX$. We define $\cO_\Delta := \Delta_* (\cO_\cX)$.
\end{defn}

To mimic Olander's proof, we need the following strengthening of \cite{BallardFavero}*{Lem. 2.16}. 

\begin{lem}
\label{L:dimbound} 
Let $\cX$ be a smooth, tame, and proper stack with quasi-projective coarse moduli space, then $\rdim(\cX) \leq \ddim(X) \leq 2 \dim(\cX)$. 
\end{lem}

\begin{proof}
    Though the result is stated for schemes, the proof of \cite{BallardFavero}*{Lem. 2.16} uses only the resolution property and that $\cO_{\Delta}$ is the Fourier-Mukai kernel of the identity functor. For $\cX$ an orbifold curve, the resolution property holds since $\cX$ is a global quotient \cite{Totaroresolution} and $\Delta_*\cO_{\cX}$ is the Fourier-Mukai kernel of the identity for formal reasons; see \cite{HuybrechtsFM}*{p. 114}.
\end{proof}

In particular, for $\cX$ an orbifold curve by \Cref{T:Orlovtheorem} we have $\ddim(\cX) \in \{ 1,2 \}$. The proof now proceeds in two parts. In the first, Olander shows that if $\cO_\Delta$ is a summand of the cone of a map $E^\bullet \boxtimes F^\bullet \to G^\bullet \boxtimes H^\bullet$, with $E^\bullet,F^\bullet,G^\bullet,H^\bullet \in \DCoh(\cX)$, then $\cO_\Delta$ is a summand of $\Cone(E\boxtimes F \to G\boxtimes H)$ for {\em vector bundles} $E,F,G,H \in\Coh(\cX)$. Olander's proof uses only the following properties, all of which hold in our setting:

\begin{enumerate}
    \item $\cO_\Delta = \bf{R}\Delta_* \cO_\cX$. Since $\cX$ has a presentation as a global quotient \cite{Kresch}, $\Delta$ is affine by \cite{Totaroresolution}*{Prop. 1.3} and hence $\Delta_*$ is exact, i.e. $\Delta_* = \bf{R}\Delta_*$
    \item For all $E \in \DCoh(\cX)$, $E \cong \bigoplus_i H^i(E)[-i]$ by \Cref{L:summand}. 
    \item By \Cref{L: boxProd} (below) , K\"unneth type objects $E \boxtimes F$ for $E, F \in \DCoh(\cX)$ decompose as sums of their shifted cohomology objects.
    \item In $\Coh(\cX$), there are no maps from a torsion sheaf to a locally free sheaf. This can be verified after passing to an \'{e}tale atlas, where the result is standard.
    \item The K\"unneth formula for sheaf cohomology on $\cX\times \cX$ holds in the form of \cite{stacks-project}*{\href{https://stacks.math.columbia.edu/tag/0BEC}{0BEC}}, i.e. for $F,G\in \QCoh(\cX)$ one obtains $H^n(\cX\times \cX, F\boxtimes G) = \bigoplus_{p+q = n} H^p(\cX,F)\otimes_k H^q(\cX,G)$ for all $n \ge 0$.
    \item $\cO_\Delta$ is not a summand of a sheaf with zero-dimensional support. This follows from the fact that the dimension of the support of $\cO_{\Delta}$ is one.
\end{enumerate}

Thus, for any orbifold curve of diagonal dimension $1$, $\cO_\Delta$ is necessarily a summand of the cone of K\"unneth type vector bundles on $\cX\times \cX$. For classical curves, Olander proves that this is impossible when $g(X) \geq 1$, and since $\ddim(\mathbb{P}^1) = 1$ the classification follows.\footnote{Strictly speaking, Olander proves more since he does not use the assumption that $k$ is algebraically closed. In particular, he proves the result for one-dimensional Brauer-Severi varieties \cite{olander2024diagonal}*{Prop. 11}.} 

\begin{lem}
\label{L: boxProd}
Let $\cX$ be a stacky curve, and $F,G$ coherent sheaves on $\cX$. Denote by $p_1, p_2$ the two projections $\cX \times \cX \to \cX$. Then 
\[
    p_1^* F \otimes p_2^* G \cong p_1^*F\otimes^{\bf{L}}  p_2^*G \in \DCoh(\cX \times \cX).
\]
\end{lem}

\begin{proof}
    Let $\pi: C \to \cX$ be an atlas of $\cX$ as in \Cref{R:atlasbycurve}. Then $C \times C \to \cX \times \cX$ is an atlas for $\cX\times \cX$, where we also denote the two projection maps for $C$ by $p_i$. Since $C\times C \to \cX\times \cX$ is faithfully flat it suffices to prove $p_1^*F\otimes p_2^*G\cong p_1^*F\otimes^{\bf{L}} p_2^*G$ in $\DCoh(C\times C)$, which is equivalent to $H^i(p_1^*F\otimes^{\bf{L}} p_2^*G) = 0$ for all $i\ne 0$. Now, \cite{olander2024diagonal}*{Lem. 3} implies the result. 
\end{proof}

We record the following adjunctions, to be used in the proof of \Cref{T:diagDimOrbCurve}.

\begin{lem}
\label{L: diagonaladjunctions}
    Let $E, F \in \DCoh(X)$. Then 
\[
    \RHom(\cO_{\Delta}, E \boxtimes F) = \mathbf{R}\Gamma(\cX, E \otimes_{\cO_{\cX}}^{\bf{L}} F \otimes_{\cO_{\cX}}^{\bf{L}} T_{\cX})[-1],
\]
    and 
\[
    \RHom(E \boxtimes F, \cO_{\Delta})= \RHom(E \otimes_{\cO_{\cX}}^{\bf{L}} F, \cO_{\cX} ).
\]
\end{lem}

\begin{proof}
Consider $\Delta_* = \mathbf{R}\Delta_* : \DCoh(\cX) \to \DCoh(\cX\times \cX) $. We have adjunctions $\mathbf{L} \Delta^* \dashv \Delta_* \dashv \Delta^!$ where $\Delta^! = S_\cX \circ \mathbf{L}\Delta^* \circ S_{\cX \times \cX}^{-1}$. Then
\begin{align*}
\RHom(\Delta_* \cO_\cX , F \boxtimes G) &\cong \RHom(\cO_\cX, \Delta^!(F \boxtimes G) )\\
&= \RHom(\cO_\cX, S_\cX( \mathbf{L}\Delta^* (F \boxtimes G \otimes \omega_{\cX \times \cX}^\vee [-2]) ))\\
&\cong \bf{R}\Gamma(\cX, \mathbf{L}\Delta^* (F \boxtimes G \otimes \omega_{\cX \times \cX}^\vee) \otimes \omega_\cX [-1] )
\end{align*}
But now $\mathbf{L}\Delta^* (F \boxtimes G \otimes \omega_{\cX \times \cX}^\vee) = \mathbf{L}\Delta^*(p_1^* F \otimes^\mathbf{L} p_2^* G \otimes \omega_{\cX \times \cX}^\vee ) = F \otimes^\mathbf{L} G \otimes \Delta^*(\omega_{\cX \times \cX}^\vee)$ as required, and $T_\cX := \omega_\cX^\vee \cong \Delta^*(\omega_{\cX \times \cX}^\vee)\otimes \omega_\cX$. The second formula follows immediately from the adjuction between $\mathbf{L}\Delta^*$ and $\bf{R} \Delta_*$ since $\mathbf{L}\Delta^* (E \boxtimes F) = E \otimes_{\cX}^{\mathbf{L}} F$.
\end{proof}

\begin{thm}
\label{T:diagDimOrbCurve}
Let $\cX$ be an orbifold curve, $\ddim(\cX) = 1$ if and only if $\deg(\omega_{\cX}) < 0$. Otherwise, $\ddim(\cX) = 2$.
\end{thm}

\begin{proof}
    Let $\cX$ be an orbifold curve with $\ddim(\cX) = 1$ so that $\cO_\Delta$ is a summand of a map of vector bundles $E\boxtimes F \to G\boxtimes H$. The bundles $E,F,G,H$ have HN-filtrations by \Cref{P:stability}, which we denote by $0 = E_0 \subset E_1\subset \cdots \subset E_n = E$ and analogously for $F,G,H$. If $\mu(E_i/E_{i-1}) = \mu_i$, we write $E^{\mu_i} = E_i$ to index by slope. These HN filtrations define filtrations of $E\boxtimes F$ and $G\boxtimes H$, where for $\gamma \in \bR$ we have $(E\boxtimes F)^\gamma = \sum_{\alpha +\beta \ge \gamma} E^\alpha \boxtimes F^\beta$. The associated graded pieces are written $\gr^\gamma(E\boxtimes F) = \bigoplus_{\alpha + \beta = \gamma}\gr^\alpha(E)\boxtimes \gr^\beta(F)$. We consider the following short exact sequences: 
    \begin{align*}
        0 \to \sum_{\alpha + \beta \geq \deg(\omega_\cX)} E^\alpha \boxtimes F^\beta \to E \boxtimes F \to Q \to 0 \\
        0 \to \sum_{\alpha + \beta >0} G^\alpha \boxtimes H^\beta \to G \boxtimes H \to Q^\prime \to 0. 
\end{align*}
    Olander's reduction strategy \cite{olander2024diagonal}*{Lem. 5} still applies to show that $\cO_\Delta$ is a summand of the cone of $\varphi: \sum_{\alpha +\beta \ge \deg(\omega_{\cX})} E^\alpha \boxtimes F^\beta \to Q'$. This involves showing $\Hom(\sum_{\alpha +\beta >0}G^\alpha \boxtimes H^\beta,\cO_{\Delta}) = 0 = \Hom(\cO_\Delta, Q[1])$. The first equality is exactly as in \emph{loc. cit.} The second uses the fact that $Q$ has a filtration with associated graded pieces sums of sheaves of the form $E'\boxtimes F'$ with $\mu(E') + \mu(F') < \deg(\omega_{\cX})$. \Cref{L: diagonaladjunctions} now gives $\Hom(\cO_{\Delta},E\boxtimes F[1]) = H^0(\cX,E\otimes F \otimes T_{\cX}) = 0$ for reasons of degree. The remainder of the proof proceeds exactly as in \emph{loc. cit.} except that the case of $g\ge 2$ is replaced by the condition $\deg(\omega_{\cX}) > 0$ and the case $g = 1$ is replaced by $\deg(\omega_{\cX}) = 0$.
\end{proof}

Let $(a_1,a_2,a_3) \in \bZ_{\ge 1}^3$ be given. We write $\bP^1_{(a_1,a_2,a_3)}$ for the orbifold projective line which has three (possibly) stacky points of orders $a_1,a_2,a_3$.  

\begin{cor}
\label{C:orbifoldDD}
    Let $\cX$ be an orbifold curve. $\ddim(\cX) = 1$ if and only if $\cX = \bP^1_{(a_1,a_2,a_3)}$ for $(a_1,a_2,a_3) \in \{(1,p,q),(2,2,r),(2,3,3),(2,3,4),(2,3,5)\}$ where $p,q\ge 1$ and $r\ge 2$. 
\end{cor}

\begin{proof}
    By \cite{Taams}, $\deg(\omega_{\cX}) = 2g(X) - 2 +\sum_{i=1}^n \frac{e_i-1}{e_i}$ and can compute that $\deg(\omega_{\cX}) < 0$ is possible only when there are at most three stacky points. Explicit calculation shows that the triples in the statement are the only ones possible. Now, by \cite{GL}*{\S 5.4.1} any such orbifold projective line is derived equivalent to the category of representations of an acyclic quiver (of extended Dynkin type). By \cite{elagin2022calculating}*{Prop. 4.5} the result now follows.
\end{proof}

\begin{cor}
\label{C:hereditarytilt}
    Let $\cX$ be an orbifold curve. $\ddim(\cX) = 1$ if and only if $\cX$ admits a tilting bundle $E$ whose endomorphism algebra is hereditary.
\end{cor}

\begin{proof}
    \cite{GL}*{\S 5.4.1} constructs an exact equivalence $\DCoh(\bP^1_{(a_1,a_2,a_3)}) \simeq \db(\mathrm{rep}\:Q_{(a_1,a_2,a_3)})$ where $Q_{(a_1,a_2,a_3)}$ is an acyclic quiver. The path algebra of the quiver arises as the endomorphism algebra of a tilting bundle coming from a strong exceptional collection \cites{AbdelgadirUeda,GL}.
\end{proof}

\section{Serre dimension and global dimension of orbifold curves}

We begin by computing the Serre dimension of orbifold curves. In fact, the usual proof that $\mathrm{Sdim}(\DCoh(X)) = \dim X$ for a smooth and projective variety $X$ works here without any major modification and we state results in greater generality than we need.

\begin{lem}
\label{L:sdimab}
    Let $\cA$ denote an Abelian category of finite homological dimension and consider its bounded derived category, $\db(\cA)$. Suppose further that 
    \begin{enumerate} 
        \item $\db(\cA)$ has a Serre functor $S$; and
        \item a strong generator $G$.
    \end{enumerate} 
    If there exist $n\in \bZ$ and $m\in \bN$ such that $S^m(\cA) \subset \cA[n]$ then 
    \[
        \lsdim(\db(\cA)) = \usdim(\db(\cA)) = \frac{n}{m}.
    \]
\end{lem}

\begin{proof}
    Up to replacing it with the sum of its cohomology objects, we may assume that $G$ lies in the heart $\cA$ of the standard t-structure. We compute $e_-(G,S^{mk}(G))$ for all $k\ge 0$. $S^{mk}(G) \in \cA[nk]$ and so $\Hom^i(G,S^{mk}(G)) \ne 0$ implies $i \ge -nk$. So, $\usdim(\db(\cA)) \le n/m$. Next, $\Hom^i(G,S^{mk}(G))\ne 0$ implies $\Ext^{i+nk}(G,G) \ne 0$ and thus $i+nk \le \rm{hd}(\cA)$. So, $-i\ge nk - \rm{hd}(\cA)$. Dividing through by $mk$ and taking $k\to \infty$ gives $\lsdim(\db(\cA)) \ge n/m$.
\end{proof}

\begin{cor}
\label{C:sdualitydm}
    If $\cX$ is a smooth and projective Deligne-Mumford stack over a field $k$ then $\usdim(\db(\cX)) = \lsdim(\db(\cX)) = n$.
\end{cor} 

\begin{proof}
    By \cite{GeometricitiyBerghLuntsSchnurer}, $\DCoh(\cX)$ has Serre functor $-\otimes \omega_{\cX}[n]$, where $\omega_{\cX} \in \Coh(\cX)$ is the dualizing sheaf. Furthermore, $\DCoh(\cX)$ has finite Rouquier dimension by \cite{BallardFavero}*{Lem. 2.20}. The result now follows from \Cref{L:sdimab}.
\end{proof}

Here, there are no hypotheses on the field $k$. \Cref{C:sdualitydm} implies in particular that for a stacky curve $\cX$ over a field $k$ the Serre dimension is one. Next, we consider global dimension in the sense of \cite{Ikeda_Qiu_2023}. First, we must show that $\Stab(\cX)\ne \varnothing$ for a projective orbifold curve $\cX$. For simplicity, we work over $\bC$. Since $K_0(\cX)$ is in general of infinite rank by \Cref{P:K0computation}, we require our stability conditions to factor through $\ch_{\rm{orb}}:K_0(\cX)\twoheadrightarrow H^*_{\rm{CR,alg}}(\cX)$, as in \Cref{D:chorb}.

We define a canonical basis of $H^*_{\rm{CR,alg}}(\cX)\otimes_{\bZ} \bR$ by using as a basis for $H^0(X;\bZ)\oplus H^2(X;\bZ)$ the Poincar\'{e} dual classes to the fundamental class and to a point. This gives an identification $H^*_{\rm{CR,alg}}(\cX)\otimes_{\bZ}\bR \cong \bR^{2+N}$ where $N = \sum_{i=1}^n e_i - 1$. Define $\lVert\:\cdot\:\rVert:H^*_{\rm{CR,alg}}(\cX)\otimes_{\bZ} \bR\to \bR$ to be the Euclidean norm with respect to this identification. Consider $\beta \in \bR$ and $H\in \bR_{>0}$. We define a function $Z_{\beta,H}:\Ob(\Coh(\cX)) \to \bC$ by $Z_{\beta,H}(E) = -\deg(E) + (\beta + iH)\rank(E)$, where $\deg(E)$ is as in \Cref{D:degree}.

\begin{thm}
\label{T:stabcond}
    For any $\beta \in \bR$ and $H\in \bR_{>0}$, the pair $\sigma_{\beta,H} = (\Coh(\cX),Z_{\beta,H})$ defines a Bridgeland stability condition on $\DCoh(\cX)$ which satisfies the support property for the orbifold Chern character map $\ch_{\rm{orb}}:K_0(\cX) \twoheadrightarrow H^*_{\rm{CR,alg}}(\cX)$. 
\end{thm}

\begin{proof}
    By \Cref{D:chorb}, the orbifold Chern character map $\ch_{\rm{orb}}:K_0(\cX) \twoheadrightarrow H^*_{\rm{CR,alg}}(\cX) = H^0(X;\bZ)\oplus H^2(X;\bZ) \oplus \bigoplus_{i=1}^n \bZ^{e_i-1}$ is given by $E\mapsto (\deg(\pi_*E),\rk(E),\mathfrak{m}_{p_1}(E),\ldots, \mathfrak{m}_{p_n}(E))$ and it follows that $Z_{\beta,H}$ factors through $\ch_{\rm{orb}}$. The fact that $(\Coh(\cX),Z_{\beta,H})$ defines a pre-stability condition follows from \Cref{P:stability} and \cite{Br07}*{Prop. 5.3}. It remains to verify the support property. It suffices to verify this for $\beta = 0$ and $H = 1$ since the other $\sigma_{\beta,H}$ are in the $\GL_2^+(\bR)^{\sim}$ orbit of $\sigma_{0,1}$; see \cite{Br07}*{Lem. 8.2} and \cite{Bayer_short}*{p.5}. So, put $Z = Z_{0,1}$. It suffices to prove that 
    \[
        \inf \left\{\frac{\lvert Z(\ch_{\rm{orb}}(E))\rvert}{\lVert \ch_{\rm{orb}}(E)\rVert} : E\in \Coh(\cX)\text{ semistable}\right\} > 0.
    \]
    So, we verify that $\lvert Z(x,y)\rvert /\lVert (x,y)\rVert$ is bounded below for all $x\in H^0(X;\bZ)\oplus H^2(X;\bZ)$ and $y\in \bigoplus_{i=1}^n \bZ^{e_i-1}$ which are non-negative, in the sense that they have all non-negative coefficients with respect to the standard basis. Write $x = x_0 + x_2$ and $y = \vec{y}_1 + \cdots + \vec{y}_n$. By the hypothesis that $(x,y)$ is non-negative, we have 
    \[
        \lvert Z(x,y) \rvert^2 \ge x_0^2 + x_2^2 + \sum_{i=1}^n \frac{\lVert \vec{y}_i \rVert^2}{e_i^2}
    \]
    As a consequence, for all $E$ semistable with respect to $\sigma_{0,1}$ in $\Coh(\cX)$
    \[
        \frac{\lvert Z(\ch_{\rm{orb}}(E))\rvert^2}{\lVert \ch_{\rm{orb}}(E)\rVert^2} \ge  \min \left\{\frac{1}{e_i^2}: 1\le i \le n\right\}.
    \]
    Thus, the support property holds and $(\Coh(\cX),Z_{\beta,H})$ defines a stability condition.
\end{proof}

\begin{cor}
\label{C:orbifoldcurvegldim}
    If $\cX$ is an orbifold curve over $\bC$ then $\gldim(\cX) = 1$.
\end{cor}

\begin{proof}
    We mimic the proof of \cite{KOT}*{Lem. 5.15}. By \Cref{C:sdualitydm} and \Cref{T:KOT} we know that $\gldim(\cX) \ge 1$. Next, by \cite{KOT}*{Prop. 4.3} given $\epsilon>0$ it will suffice to construct $\sigma_\epsilon = (Z,\cP) \in \Stab(\cX)$ such that for any $\phi \in \bR$ one has $S(\cP(\phi)) \subset \cP[\phi+1-\epsilon, \phi+ 1 +\epsilon]$; here, $S = -\otimes \omega_{\cX}[1]$ is the Serre functor on $\DCoh(\cX)$. Put $d = \deg(\omega_{\cX}) = 2g - 2 + \sum_i \tfrac{e_i-1}{e_i}$. $\mathrm{arccot}:\bR \to (0,\pi)$ is uniformly continuous so there exists an $H>0$ such that 
    \[
        \frac{1}{\pi}\left\lvert \mathrm{arccot}\left(x - \tfrac{d}{H}\right) - \mathrm{arccot}\left( x\right)\right\rvert < \epsilon
    \]
    for all $x\in \bR$. Put $\sigma_\epsilon = \sigma_{0,H}$ for some such $H$. Now, if $E$ is $\sigma_\epsilon$-semistable then $\rk(E) = 0$ is equivalent to $E$ being torsion and thus so is $E\otimes \omega_{\cX}$. It follows that $\phi_{\sigma_\epsilon}(E\otimes \omega_{\cX}) = \phi_{\sigma_\epsilon}(E)$. If $\rk(E) \ne 0$, then $Z_{0,H}(E\otimes \omega_{\cX}) = Z_{0,H}(E) - d\cdot \rk(E)$. Thus, one has 
    \[
        \lvert \phi_{\sigma_\epsilon}(E\otimes \omega_{\cX}) - \phi_{\sigma_\epsilon}(E)\rvert = \frac{1}{\pi}\left\lvert \mathrm{arccot}\left(\tfrac{-\deg(E)}{H\cdot \rk(E)} - \tfrac{d}{H}\right) - \mathrm{arccot}\left(\tfrac{-\deg(E)}{H\cdot \rk(E)}\right) \right\rvert < \epsilon.
    \]
    Thus, $S(\cP_{\sigma_\epsilon}(\phi)) \subset \cP_{\sigma_\epsilon}[\phi+1-\epsilon, \phi+1+\epsilon]$.
\end{proof}

\begin{rem}
    We have constructed stability conditions in \Cref{T:stabcond} only in the case of $\cX$ defined over $\bC$. The main reason for this restriction is to have a convenient topological definition of the charge lattice $\Lambda$ through which the central charge factors. In the case where one prefers to work over a general base field, one can use the numerical Grothendieck group $\cN(X)$ in lieu of $H^0(X;\bZ)\oplus H^2(X;\bZ)$. The proof carries through \emph{mutatis mutandis}. On the other hand, \Cref{C:orbifoldcurvegldim} makes use of \Cref{T:KOT} which restricts to the case of the ground field being $\bC$. It seems unlikely that this restriction is necessary, but at present we have not investigated this.
\end{rem}

\appendix

\section{Homological Lemma}

For completeness we provide a proof of the following lemma, used without proof in \cite{orlov2008remarks}, and to be used in \Cref{T:Orlovtheorem}. As such, we have in mind the case where $F$ is a vector bundle on an orbifold curve $\cX$ with sub-bundle $F_i \subseteq F$ and such that $F_i$ and $F/F_i$ fit into the exact sequences from \Cref{L:Orlovmainlemma}:
\begin{align*}
        (\pi^*L^{-1}\otimes \cE)^{\oplus r_1} &\xrightarrow{\alpha} \cE^{\oplus r_0}\to F_i\to 0 \\
        0 \to F/F_i &\to (\pi^*L\otimes \cE^\vee)^{\oplus s_0} \xrightarrow{\beta} (\pi^*L^2\otimes \cE^\vee)^{\oplus s_1}.
\end{align*}

\begin{lem}
\label{L: orlovCone}
Let $\cA$ an Abelian category of homological dimension one with enough injectives. Let $F \in \cA$ with subobject $F_i \subseteq F$, and suppose we have the following exact sequences in $\cA$
    \begin{align}
        L_{-1} &\xrightarrow{\alpha} L_0 \to F_i\to 0 \label{E:(6)}\\
        0 \to F/F_i &\to L_1 \xrightarrow{\beta} L_2\label{E:(7)}
    \end{align}
    for some objects $L_{-1}, L_0, L_1, L_2 \in \cA$. Then $F$ is a summand of $\Cone(\phi) \in \rm{D}(\cA)$, where $\phi$ is the map  
    \[
    \begin{tikzcd} 
        L_{-1} \oplus L_1[-1] \arrow[r,"\sbm{\alpha \amp \varepsilon\\ 0\amp \beta}"] &
        L_0 \oplus L_2 [-1] 
    \end{tikzcd}
    \]
\end{lem} 

\begin{proof} 
First note that $0 \to F_i \to F \to F / F_i \to 0$ determines a morphism $e:F / F_i [-1] \to F_i$ such that $\Cone(e) \cong F$. Writing $\varphi:0\to \ker(\alpha)[1]$ and $\psi:\coker(\beta)[-2]\to 0$, we have $e\oplus \varphi \oplus \psi: F/F_i[-1]\oplus \coker(\beta)[-2] \to F_i \oplus \ker(\alpha)[-1]$ with $\Cone(e\oplus \varphi \oplus \psi) = F\oplus \coker(\beta)[-1]\oplus \ker(\alpha)[1]$. Using \Cref{L:summand}, \eqref{E:(6)}, and \eqref{E:(7)} we have quasi-isomorphisms 
\[ 
    (L_{-1}\xrightarrow{\alpha}L_0) \simeq F_i[-1] \oplus \ker(\alpha) \quad \text{ and }\quad (L_1\xrightarrow{\beta}L_2) \simeq F/F_i\oplus \coker(\beta)[-1].
\]
In particular, $e\oplus \varphi \oplus \psi$ induces a map $e^\prime \in \Hom(\beta[-1], \alpha[1])$ with isomorphic cone. We will compute $\Cone(e')$ by injectively replacing $(L_{-1}\xrightarrow{\alpha} L_0)[1]$. First, we construct a commutative diagram:
\begin{equation}
\label{E:resolution}
\begin{tikzcd}
    0 \arrow[r]&L_{-1}\arrow[r,"i"]\arrow[d,"\alpha"]&I^0\arrow[d,"\alpha^0"]\arrow[r,"d_I"]&I^1\arrow[d,"\alpha^1"]\\
    0\arrow[r]&L_0\arrow[r,"j"]&J^0\arrow[r,"d_J"]&J^1
\end{tikzcd}
\end{equation}
where the horizontal rows are injective resolutions. Here, we have used $\rm{hd}(\cA) = 1$ to obtain length two resolutions. We now have an injective replacement:
\[
\begin{tikzcd}
    L_{-1}\arrow[r,"\alpha"]\arrow[d,"i"]&L_0\arrow[d,"\sbm{0\\ j}"] \\
    I^0\arrow[r,"\sbm{d_I\\ -\alpha^0}"]&I^1\oplus J^0\arrow[r,"\sbm{\alpha^1\amp d_J}"]&J^1.
\end{tikzcd}
\]
Now, $-e'$ can be lifted to a morphism of complexes:

\begin{equation}
\label{E:complex1}
    \begin{tikzcd}
        &&L_1\arrow[r,"\beta"]\arrow[d,"-e''"]&L_2\\
        I^{0}\arrow[r]&I^1\oplus J^0 \arrow[r]&J^1.&
    \end{tikzcd}
\end{equation}
By definition, the mapping cone of \eqref{E:complex1} is: 
\begin{equation}
\label{E:mappingconeofcomplex1}
    \begin{tikzcd}
        I^0\arrow[r,"\sbm{0\\d_I\\ -\alpha^0}"]
        &L_1\oplus I^1\oplus J^0 \arrow[rr,"\sbm{-\beta \amp 0 \amp 0 \\
        -e''\amp \alpha^1\amp d_J}"]&&L_2\oplus J^1
    \end{tikzcd}
\end{equation}  
Applying the injective resolutions of $L_{-1}$ and $L_0$ as in \eqref{E:resolution} to the proposed map $\phi$ from the lemma statement, we get a (vertical) morphism of complexes:
\begin{equation}
\label{E:statementmap}
\begin{tikzcd}
    I^0\arrow[r,"\sbm{0\\d_I}"]\arrow[d,"\alpha^0",swap]&L_1\oplus I^1\arrow[d,"\sbm{\beta\amp0\\ e''\amp \alpha^1}"]\\
    J^0\arrow[r,"\sbm{0\\d_J}"]&L_2\oplus J^1.
\end{tikzcd}
\end{equation}
The (vertical) cone of \eqref{E:statementmap} is the first row of \eqref{E:lastdiagram} and an isomorphism of this cone with \eqref{E:mappingconeofcomplex1} is given by the vertical arrows:
\begin{equation}
\label{E:lastdiagram}
    \begin{tikzcd}
        I^0\arrow[dd,"-\id",swap]\arrow[r,"\sbm{0\\-d_I\\ \alpha^0}"]
        &L_1\oplus I^1\oplus J^0 \arrow[dd,"\sbm{-1\amp\amp\\ \amp 1\amp \\ \amp \amp 1}"]
        \arrow[rr,"\sbm{\beta \amp 0 \amp 0 \\
        e''\amp \alpha^1\amp d_J}"]&&L_2\oplus J^1\arrow[dd,"\id"]\\
        &&\\
        I^0\arrow[r,swap,"\sbm{0\\d_I\\ -\alpha^0}"]
        &L_1\oplus I^1\oplus J^0 \arrow[rr,swap,"\sbm{-\beta \amp 0 \amp 0 \\
        -e''\amp \alpha^1\amp d_J}"]&&L_2\oplus J^1.
    \end{tikzcd}
\end{equation}
In particular, it follows that the cone of \eqref{E:statementmap}, which computes $\Cone(\phi)$, is isomorphic to $\Cone(-e'')$, and thus to $\Cone(e')$. Therefore, it contains $F$ as a summand, as claimed.
\end{proof}

\section{Stability lemmas}

We collect some facts about slope stability on $\cX$ that are used in \Cref{L:Orlovmainlemma}(b). For this section, let $F$ denote a vector bundle on $\cX$ and $0 = F_0\subset F_1\subset\cdots \subset F_n = F$ its HN filtration as in \Cref{P:stability}. The sequence of inclusions dualizes to a sequence of surjections 
\[
    F^\vee = F_n^\vee \twoheadrightarrow \cdots \twoheadrightarrow F_1^\vee \twoheadrightarrow F_0^\vee =0.
\]
Put $K_i = \ker(F^\vee \twoheadrightarrow F_i^\vee)$. We obtain a descending filtration $0 = K_n\subset K_{n-1}\subset \cdots \subset K_1\subset K_0 = F^\vee$.

\begin{lem}
\label{L:dualHN}
    $0 = K_n\subset K_{n-1}\subset \cdots \subset K_1\subset K_0 = F^\vee$ is the HN filtration of $F^\vee$.
\end{lem}

\begin{proof}
    First, note that the dual of a semistable bundle $E$ is semistable. Indeed, given $G \subset E^\vee$ we get a surjection $E\twoheadrightarrow G^\vee$ and since $E$ is semistable one has $\mu(E) \le \mu(G^\vee)$. Therefore, $\mu(G)\le \mu(E^\vee)$. Next, consider the following commutative diagram
    \[
    \begin{tikzcd}
        &&&0\arrow[d]&\\
        &&&F_{i+1}/F_i\arrow[d,hook]&\\ 
        0 \arrow[r]& F_i\arrow[d,hook] \arrow[r]& F\arrow[d,equal] \arrow[r]& F/F_i \arrow[r]\arrow[d,two heads]& 0\\
        0 \arrow[r]& F_{i+1} \arrow[r]& F\arrow[r] & F/F_{i+1}\arrow[r]\arrow[d] & 0 \\
        &&&0&
    \end{tikzcd}
    \]
    Dualising gives $K_i/K_{i+1} \cong (F_{i+1}/F_i)^\vee$, which is a semistable vector bundle with slope $-\mu(F_{i+1}/F_i)$. Write $G_i = K_{i-1}/K_i$. It follows that $\mu(G_n)>\cdots > \mu(G_1)$. 
\end{proof}

\begin{lem}
\label{L:Orlovlemmab}
    \Cref{L:Orlovmainlemma}(b) follows from \Cref{L:Orlovmainlemma}(a). 
\end{lem}

\begin{proof}
    Using \Cref{L:dualHN}, one computes that \Cref{L:Orlovmainlemma}(a) applies to $K_i \otimes \pi^*L$. As $K_i = (F/F_i)^\vee$, we obtain $(\pi^*L^{-1}\otimes \cE)^{\oplus s_1} \to \cE^{\oplus s_0}\to (F/F_i)^\vee \otimes \pi^*L \to 0$ and dualizing this gives (b) as claimed. 
\end{proof}

\bibliography{bibl}

\end{document}